\documentclass[12pt,reqno]{amsart}
\advance \topmargin by -\headheight
\advance \topmargin by -\headsep
\evensidemargin \oddsidemargin
\usepackage{amsmath}
\usepackage{amsfonts}
\usepackage{stmaryrd}
\usepackage{amssymb}
\usepackage{amsthm}
\usepackage{csquotes}

\usepackage{mathrsfs}
\usepackage{dsfont}
\usepackage{bm}
\usepackage{cite}
\usepackage{soul}

\numberwithin{equation}{section}
\renewcommand\vec{\bm}
\newcommand{\n}[1]{\|{#1}\|}
\newcommand\ls{\lessapprox}  
\newcommand\gs{\gtrapprox}    

\newtheorem{theorem}{Theorem}[section]

\newtheorem{lemma}[theorem]{Lemma}
\newtheorem{Proposition}[theorem]{Proposition}
\newtheorem{Conjecture}[theorem]{Conjecture}
\newtheorem{Corollary}[theorem]{Corollary}

\usepackage{mathtools}
\DeclarePairedDelimiter{\ceil}{\lceil}{\rceil}

\title[Diameter free estimates for Quadratic Vinogradov Systems]{Diameter free estimates for the \\ Quadratic Vinogradov Mean Value Theorem}
\author[Akshat Mudgal]{Akshat Mudgal}
\date{} 
\address{Mathematical Institute, University of Oxford, Radcliffe Observatory Quarter, Woodstock Road, Oxford OX2 6GG, UK}
\email{Akshat.Mudgal@maths.ox.ac.uk}
\renewcommand\vec{\bm}

\begin{document}

\maketitle
\begin{abstract}
Let $s \geq 3$ be a natural number, let $\psi(x)$ be a polynomial with real coefficients and degree $d \geq 2$, and let $A$ be some large, non-empty, finite subset of real numbers. We use $E_{s,2}(A)$ to denote the number of solutions to the system of equations 
\[ \sum_{i=1}^{s} (\psi(x_i) - \psi(x_{i+s}) )= \sum_{i=1}^{s} ( x_i - x_{i+s} ) = 0, \]
where $x_i \in A$ for each $1 \leq i \leq 2s$. Our main result shows that
\[ E_{s,2}(A) \ll_{d,s} |A|^{2s -3 + \eta_{s}}, \]
where $\eta_3 = 1/2$, and $\eta_{s} = (1/4- 1/7246)\cdot 2^{-s + 4}$ when $s \geq 4$. The only other previously known result of this flavour is due to Bourgain and Demeter, who showed that when $\psi(x) = x^2$ and $s=3$, we have 
\[E_{3,2}(A) \ll_{\epsilon} |A|^{3 + 1/2 + \epsilon},\]
 for each $\epsilon > 0$. Thus our main result improves upon the above estimate, while also generalising it for larger values of $s$ and more wide-ranging choices of $\psi(x)$.
\par

The novelty of our estimates is that they only depend on $d$, $s$ and $|A|$, and are independent of the diameter of $A$. Thus when $A$ is a sparse set, our results are stronger than the corresponding bounds that are provided by methods such as decoupling and efficient congruencing. Consequently, our strategy differs from these two lines of approach, and we employ techniques from incidence geometry, arithmetic combinatorics and analytic number theory. Amongst other applications, our estimates lead to stronger discrete restriction estimates for sparse sequences. 
\end{abstract}
%



\section{Introduction}

Let $s \geq 3$ be a natural number, let $\psi(x)$ be a polynomial with real coefficients and degree $d \geq 2$, and let $A$ be a finite, non-empty subset of real numbers. In this paper, we study the number of solutions to the system of equations
\begin{equation}  \label{sys1}
\sum_{i=1}^{s} (\psi(x_i) - \psi(x_{i+s}) )= \sum_{i=1}^{s} ( x_i - x_{i+s} ) = 0 ,
\end{equation}
such that $x_i \in A$ for each $1 \leq i \leq 2s$. In particular, we define $E_{s,2}(A)$ to be the number of $2s$-tuples $(x_1, \dots, x_{2s}) \in A^{2s}$ such that $x_1, \dots, x_{2s}$ satisfy $\eqref{sys1}$. Our main result presents upper bounds for $E_{s,2}(A)$ in terms of $s$, $d$ and the cardinality of $A$.

\begin{theorem} \label{mainth1}
Let $s$ be a natural number, let $\psi(x)$ be a polynomial with real coefficients and degree $d \geq 2$, let $A$ be a finite, non-empty subset of real numbers. When $s\geq 3$, we have
\[ E_{s,2}(A) \ll_{d,s} |A|^{2s - 3+ \eta_{s}}, \]
where $\eta_3 = 1/2$, and $\eta_{s} = (1/4 - 1/7246)\cdot 2^{-s+4}$ for $s \geq 4$.
\end{theorem}

We now present a straightforward application of our theorem. We begin by setting $\psi(x) = x^2$ in $\eqref{sys1}$ to obtain the quadratic Vinogradov system 
\begin{equation}  \label{qvin}
\sum_{i=1}^{s} (x_i^2 - y_i^2) = \sum_{i=1}^{s} (x_i - y_i) = 0.
\end{equation}
This system and its generalisations have been widely studied, and in particular, finding suitable upper bounds for the number of solutions to these type of systems has been a major topic of work. Given a finite, non-empty set $A$ of real numbers, we use $J_{s,2}(A)$ to denote the number of solutions to $\eqref{qvin}$ with $x_i, y_i \in A$ for each $1 \leq i \leq s$. Theorem $\ref{mainth1}$ implies the following corollary. 

\begin{Corollary} \label{wait}
Let $s \geq 3$ be a natural number and let $A$ be a finite subset of real numbers. Then we have
\[ J_{s,2}(A) \ll_{s} |A|^{2s - 3 + \eta_{s}}, \]
where $\eta_3 = 1/2$, and $\eta_{s} = (1/4 - 1/7246)\cdot 2^{-s+4}$ for $s \geq 4$.
\end{Corollary}

To put this in context, the only other previously known result of this flavour is due to Bourgain and Demeter \cite[Proposition 2.15]{BD2015} who showed that 
\[ J_{3,2}(A) \ll_{\epsilon} |A|^{3 + 1/2 + \epsilon} , \]
for each $\epsilon > 0$. Thus, Theorem $\ref{mainth1}$ provides improvement on their result, while also generalising it for larger values of $s$ and more wide-ranging choices of $\psi(x)$. Moreover, when $s \geq 3$, there exist arbitrarily large sets $A$ of integers such that 
\[ J_{s,2}(A) \gg_{s} |A|^{2s - 3}.\]
Hence, for large values of $s$, Corollary $\ref{wait}$ provides estimates that come arbitrarily close to the conjectured upper bounds for $J_{s,2}(A)$. 
\par

Furthermore, one may contrast Corollary $\ref{wait}$ with estimates arising from decoupling (see \cite{BD2015}, \cite{BDG2016}) and efficient congruencing (see \cite{Wo2017}). In order to present a quick comparison, we focus on the case when $A$ is a set of integers, whereupon, a result of Bourgain and Demeter \cite{BD2015} implies that
\begin{equation} \label{per}
 J_{s,2}(A) \ll_{s,\epsilon} X_A^{\epsilon} |A|^{2s-3}
 \end{equation}
for each $\epsilon>0$, with $X_A = \sup A - \inf A$ denoting the diameter of $A$. Thus, when $X_A$ is small in terms of $|A|$, say, $X_A \leq |A|^{m}$ for some fixed positive constant $m$, these estimates are almost sharp. But in the case when $X_A$ is large, say, when $X_A \geq 2^{2^{|A|}}$, these bounds perform worse than the trivial estimates. Since our results are independent of the diameter $X_A$ of $A$, we note that whenever $A$ is a sparse set, Theorem $\ref{mainth1}$ provides better estimates than inequalities of the form $\eqref{per}$. We describe these comparisons in more detail later in this section.
\par

We further mention that when $s \geq 4$, we have $\eta_{s} < 2^{-s+2}$. An estimate analogous to Theorem $\ref{mainth1}$ with $\eta_{s}$ replaced by $2^{-s+2}$ can be deduced using purely incidence geometric methods (see Theorem $\ref{main}$). Moreover, obtaining a power saving over the exponent $2^{-s+2}$ requires utilising multiple techniques from incidence geometry, arithmetic combinatorics and analytic number theory. Thus, Theorem $\ref{mainth1}$ can be interpreted as a threshold breaking result. This is explained in further detail, along with some additional applications of our methods, in \S2.

The quantity $E_{s,2}(A)$ can be interpreted as the $s$-fold additive energy of sets of finite points lying on the curve $y = \psi(x)$, that is, $E_{s,2}(A)$ counts the number of solutions to the equation
\[ x_1 + \dots + x_s =  x_{s+1} + \dots + x_{2s}, \]
such that the variables $x_1, \dots, x_{2s}$ lie in the set $\mathscr{A}$, where
\[ \mathscr{A} = \{ (a, \psi(a)) \ | \ a \in A \}. \]
Thus, we introduce the related notion of a sumset. For each $s,t \in \mathbb{N}$, we define
\[ s \mathscr{A} = \{ \vec{a}_1 + \dots + \vec{a}_{s} \ | \  \vec{a}_1, \dots, \vec{a}_{s} \in \mathscr{A} \}, \]
and
\[  s \mathscr{A} - t \mathscr{A} = \{ \vec{a}_1 + \dots + \vec{a}_{s} - \vec{a}_{s+1} - \dots - \vec{a}_{s+t} \ | \  \vec{a}_1, \dots, \vec{a}_{s+t} \in \mathscr{A} \}. \]
\par

These particular types of sumsets were studied in \cite{Ak2020}, and from the point of view of additive combinatorics, estimates on cardinalities of such sets are closely related to bounds for $E_{s,2}(A)$. In this paper, we record some further threshold breaking lower bounds for cardinalities of sumsets of the above form.

\begin{theorem} \label{mainth2}
Let $A$ be a finite, non-empty subset of real numbers. Then we have
\begin{equation} \label{then3}
 |2 \mathscr{A} - 2 \mathscr{A}| \gg_{d}  |A|^{3 - 2/11}(\log |A|)^{-18/11}.
 \end{equation}
Moreover, let $s \geq 3$ be a natural number. Then we have
 \begin{equation} \label{then4}
|s \mathscr{A} - s\mathscr{A}| \gg_{d,s} |A|^{ 3 - \delta \cdot 4^{-s+3} } (\log|A|)^{ -C \cdot 4^{-s+3}},
\end{equation}
where $\delta = (1 - 4c)/23$ and $c = 1/7246$ and $C = 36/23$.
\end{theorem}

As before, we remark that bounds of the shape $|2 \mathscr{A}-2\mathscr{A}| \gg_{d} |A|^{3 - 1/4}$ are attainable using purely incidence geometric methods, and it is surpassing this bound that requires additional ideas. 
\par

We note that one could also use Theorem $\ref{mainth1}$ to furnish lower bounds for these type of sumsets. In order to see this, we define, for each $\vec{n}  = (n_1, n_2)$ in $\mathbb{R}^2$, the quantity $r_{s}(\vec{n})$ that counts the number of solutions to the system of equations
\begin{equation} \label{twp}
n_1 - \sum_{i=1}^{s} x_i = n_2 - \sum_{i=1}^{s} \psi(x_i) = 0,
\end{equation}
where $x_i \in A$ for each $1 \leq i \leq s$. Using double counting, we observe that
\[  \sum_{\vec{n} \in s \mathscr{A}} r_{s}(\vec{n}) = |A|^{s} \ \text{and} \  \sum_{\vec{n} \in s \mathscr{A}} r_{s}(\vec{n})^2 = E_{s,2}(A) . \]
Thus, a straightforward application of Cauchy-Schwarz inequality implies that
\begin{equation} \label{tirs}
 E_{s,2}(A) | s \mathscr{A}| \geq |A|^{2s}, 
 \end{equation}
whenever $s \in \mathbb{N}$. Similarly, we have
\[  E_{s+t,2}(A) | s \mathscr{A} - t \mathscr{A}| \geq |A|^{2s + 2t}, \]
whenever $s,t \in \mathbb{N}$. Hence, we could deduce lower bounds of the form 
\[ |s\mathscr{A} - s\mathscr{A}| \gg |A|^{3 - \eta_{2s}}\]
directly from Theorem $\ref{mainth1}$, whenever $s \geq 2$, but in all such cases, Theorem $\ref{mainth2}$ provides stronger quantitative estimates. Moreover, as in the case of Theorem $\ref{mainth1}$, inequality $\eqref{then4}$ in Theorem $\ref{mainth2}$ misses the conjectured estimate $\eqref{condel}$ by a factor of $|A|^{\delta \cdot 4^{-s+3}}$.
\par

We further remark that while Theorems $\ref{mainth1}$ and $\ref{mainth2}$ are our main results, their proofs require many auxiliary results which are of independent interest. More specifically, throughout this paper, we derive multiple estimates for various moments of the function $r_s$, for some fixed set $A$. This includes the third moment 
\begin{equation} \label{fut} 
E_{s,3}(A) = \sum_{\vec{n} \in s \mathscr{A}} r_{s}(\vec{n})^3,
\end{equation}
that is, the $s$-fold third energy of the set $A$ (see Theorems $\ref{kigi}$ and $\ref{newthirdnew}$). We also provide non-trivial upper bounds for $r_{s}(\vec{n})$ which just depend on $s$, $d$ and $|A|$, and in particular, are uniform in $\vec{n}$ (see Theorem $\ref{kglw}$). We present these results in a more detailed manner in \S2.
\par

We now describe some of the ideas that we use to establish our results. We begin by converting the problem of finding upper bounds for $E_{s,2}(A)$ to estimating incidences between translates of the curve $y = \psi(x)$ and sets of points of the form $s \mathscr{A}$. Thus, we use weighted variants of Szemer\'{e}di-Trotter theorem. Our next aim is to consider the $s$-fold third energy $E_{s,3}(A)$ of $A$. Here, we combine ideas from incidence geometry with techniques from analytic number theory, and in particular, the theory of averaging over slim exceptional sets. The latter was first explored in \cite{Wo2002b}, and further applied in \cite{Wo2002a} and \cite{Wo2002c}. Using these methods, we are able to obtain our first set of non-trivial upper bounds for $E_{s,3}(A)$ whenever $s \geq 3$. These bounds form a crucial ingredient for applying our next set of tools, that is, the higher energy machinery. These techniques were originally proposed by Schoen and Shkredov in \cite{SS2012}, and were further developed and applied in various other works, for instance, see \cite{SS2011, SS2013, Sh2013, RS2019}. In particular, we generalise ideas from \cite{RS2019} and combine them with some more incidence geometric estimates to furnish stronger lower bounds for $|s\mathscr{A} - s\mathscr{A}|$ in terms of various third energy estimates. We use these bounds along with inverse results from additive combinatorics to obtain strengthened estimates for the second $s$-fold energy $E_{s,2}(A)$ whenever $s \geq 4$, and consequently, the third $s$-fold energy $E_{s,3}(A)$ whenever $s \geq 5$. Iterations of these results play a key role in leading up to our main estimates for $E_{s,2}(A)$ and $|s \mathscr{A} - s \mathscr{A}|$.
\par

An advantage that our method provides us is that we do not require $\eqref{sys1}$ to be translation-dilation invariant, as can be seen by the hypothesis that $\psi$ can be any generic polynomial with real coefficients and fixed degree $d \geq 2$. This contrasts with techniques such as efficient congruencing and decoupling, both of which require some form of translation-dilation invariance, or as it is otherwise known, parabolic rescaling. In fact, our methods allow us to find estimates for systems of equations of the form $\eqref{sys1}$ where $\psi$ is a more general function. Thus, let $I$ be an interval on the real line, let $\psi : I \to \mathbb{R}$ be a continuous function and let $\mathcal{C}>0$ be some parameter such that the following two conditions hold. First, for every pair $\delta_1, \delta_2$ of real numbers such that $\delta_1 \neq 0$, we have
\begin{equation} \label{fis1}
|\{ (x,y) \in I^2 \ | \ x-y = \delta_1 \ \text{and} \ \psi(x) - \psi(y) = \delta_2 \}| \ll_{\mathcal{C}} 1.
\end{equation}
Moreover, for every pair $n_1,n_2$ of real numbers, we have
\begin{equation} \label{fis2}
  |\{ (x,y) \in I^2 \ | \  x+y = n_1 \ \text{and} \ \psi(x) + \psi(y) = n_2 \}| \ll_{\mathcal{C}} 1.
  \end{equation}
In such a case, we can apply our methods, as in the case of $\psi$ being a polynomial of degree $d$, to provide appropriate estimates for $E_{s,2}(A)$ whenever $s \geq 3$. We describe this in more detail in \S2. 
\par

We return our focus to the quadratic Vinogradov system. As we noted before, Bourgain and Demeter showed that for all finite, non-empty subsets $A$ of real numbers, and for each $\epsilon > 0$, we have
\begin{equation} \label{demth}
 J_{3,2}(A) \ll_{\epsilon} |A|^{3 + 1/2 + \epsilon}.
\end{equation}
Moreover using Young's convolution inequality, we can deduce that
\begin{equation} \label{young}
 J_{s,2}(A) \leq |A|^{2s -6} J_{3,2}(A) , 
 \end{equation}
whenever $s \geq 4$. Upon combining this with $\eqref{demth}$, we obtain
\[ J_{s,2}(A) \ll_{\epsilon}  |A|^{2s - 3 + 1/2 + \epsilon}. \]
Our aim has been to improve upon the above estimate for large values of $s$. In this endeavour, we provide a series of successive improvements, beginning with Theorem $\ref{main}$ which implies that
\[ J_{s,2}(A) \ll_{s}|A|^{2s - 3 + 2^{-s+2}} \]
for each $s \geq 3$, and culminating in Theorem $\ref{mainth1}$ and Corollary $\ref{wait}$.
\par

In terms of lower bounds, if we consider the set $A_N = \{1,2,\dots, N\}$, then using $\eqref{tirs}$ and the fact that $|s \mathscr{A}_N| \ll_{s} N^{3},$ where $\mathscr{A}_N = \{ (n,n^2) \ | \ n \in A_N \}$, we can infer that
\[ J_{s,2}(A_N) \gg_{s} N^{2s-3} = |A_N|^{2s-3}. \]
Conjecturally, it is believed that for any fixed $s \geq 3$ and $\epsilon > 0$, the above example is sharp up to a factor of $|A_N|^{\epsilon}$. In particular, Bourgain and Demeter \cite[Question 2.13]{BD2015} conjectured that for all  finite, non-empty sets $A$ of real numbers and for all $\epsilon > 0$, we have
\begin{equation} \label{halfway}
J_{3,2}(A) \ll_{\epsilon} |A|^{3 + \epsilon}. 
\end{equation}
Combining this with $\eqref{young}$, we see that $\eqref{halfway}$ is equivalent to the following conjecture (see also \cite[Conjecture $1.5$]{Ak2020}).

\begin{Conjecture} \label{demcon}
Let $s \geq 3$ be a natural number, let $A$ be a finite, non-empty set of real numbers and let $\epsilon > 0$ be a real number. Then we have
\[  J_{s,2}(A) \ll_{\epsilon} |A|^{2s - 3 + \epsilon}. \]
\end{Conjecture}

Thus both Corollary $\ref{wait}$ and Theorem $\ref{main}$ come arbitrarily close to the conjectured upper bound as $s$ grows, with the former converging at a slightly faster rate. Furthermore, the above conjectured estimate may be combined with $\eqref{tirs}$ to deliver the corresponding conjectured lower bound for sumsets on the parabola, that is,
\[ |s \mathscr{A}| \gg_{\epsilon} |A|^{3 - \epsilon} \]
 for every $\epsilon >0$ and for every finite, non-empty set $A$ of real numbers, whenever $s \geq 3$. Similarly, a conjectured lower bound of the shape
 \begin{equation} \label{condel}
 |s \mathscr{A} - s\mathscr{A}| \gg_{\epsilon} |A|^{3 - \epsilon} 
 \end{equation}
 for every $s \geq 2$ and $\epsilon >0$ and finite set $A \subseteq \mathbb{R}$ may be deduced from Conjecture $\ref{demcon}$.
\par

As previously mentioned, Conjecture $\ref{demcon}$ holds true for a particular class of sets. More specifically, for a fixed $m \geq 1$, we write a set $A$ to be \emph{$m$-dense} if its diameter $X_A$ satisfies $X_A \leq |A|^{m}$. Moreover, we call a finite subset $A$ of real numbers to be \emph{well-spaced} if 
\[  | A \cap (j, j+1]| \leq 1 \ \text{for all} \ j \in [ \inf A - 1, \sup A] \cap \mathbb{Z} . \]
We observe that the quadratic Vinogradov system of equations is translation and dilation invariant. Hence, we can dilate any finite set $A$ appropriately to ensure that it is well-spaced without varying $J_{s,2}(A)$, but at the cost of its diameter $X_{A}$ becoming large. 
\par

With these definitions in hand, we note that for any fixed $m \geq 1$, if a set $A$ is well-spaced and $m$-dense, then $\eqref{per}$ implies that
\begin{equation*} 
J_{s,2}(A)\ll_{s, \epsilon} X_A^{\epsilon} |A|^{2s - 3}  \ll_{s, \epsilon}   |A|^{2s-3+ m\epsilon},
\end{equation*} 
whence, we obtain the desired conclusion by rescaling $\epsilon$ appropriately. We refer the reader to \cite{Ak2020} for a more detailed introduction to the discussion surrounding Conjecture $\ref{demcon}$ and estimates of the above form.
\par

We now give another application of Theorem $\ref{mainth1}$. It was shown in \cite{GG2019} that proving non-trivial upper bounds for $E_{s,2}(A)$, for generic sets $A$ of integers, can yield discrete restriction estimates. Thus, let $s$ be a natural number, and let $\phi : \mathbb{N} \to \mathbb{Z}^2$ satisfy $\phi(n) = (n, \psi(n))$, where $\psi$ is a polynomial of degree $d \geq 2$ with real coefficients. We note that the conclusion of Theorem $\ref{mainth1}$ combines naturally with the hypothesis of \cite[Theorem $1.1$]{GG2019} to deliver the following result.

\begin{Corollary} \label{rp}
Let $N$ be a large natural number, and let $\{ \mathfrak{a}_j \}_{j=1}^{N}$ be a sequence of complex numbers. Then whenever $s \geq 3$, we have
\[ \int_{[0,1)^{2}} |\sum_{j=1}^{N}  \mathfrak{a}_j e( \phi(j) \cdot  \vec{\alpha}) |^{2s} d \vec{\alpha} \ll_{d,s} (\log N)^{2s - q_{s}} \bigg( \sum_{j=1}^{N} | \mathfrak{a}_j|^{2s/q_{s}} \bigg)^{q_{s}},  \]
where $q_{3} = 7/2$, and $q_{s} =  2s - 3 + (1/4 - 1/7246) \cdot 2^{-s + 4}$ whenever $s \geq 4$. 
\end{Corollary}

Again, in the setting when $\phi(j) = (j,j^2)$, we compare this with a result of Bourgain \cite[Proposition $2.36$]{Bo1993} which states that there exists some $C > 0$ such that for each large enough natural number $N$ and each sequence  $\{ \mathfrak{a}_j \}_{j=1}^{N}$ of complex numbers, we have
\begin{equation} \label{BoRes}
 \int_{[0,1)^{2}} |\sum_{j=1}^{N}  \mathfrak{a}_j e( j \alpha_1 + j^2 \alpha_2) |^{6} d \vec{\alpha} \ll e^{ \frac{C\log N}{\log \log N}} \bigg( \sum_{j=1}^{N} | \mathfrak{a}_j|^{2} \bigg)^{3}. 
 \end{equation}
Since Bourgain's result \cite{Bo1993}, many techniques have been developed to work on various types of discrete restriction estimates. This includes the method of decoupling and efficient congruencing. Despite these new tools, $\eqref{BoRes}$ has not been improved upon for the situation when one requires an $l^2$ norm on the right hand side (see the remark below Theorem $1.1$ in \cite{Li2018}). 
\par

We note that Corollary $\ref{rp}$ provides a much better dependence on $N$ than $\eqref{BoRes}$, but at the cost of moving to an $l^{12/7}$ norm. To see this, we record the following example. Let $\lambda \geq 1$ be a fixed constant, let $N$ be some large natural number and let $\{ \mathfrak{a}_j \}_{j=1}^{N}$ be a sequence of complex numbers. We write
\[ S = \{ 1 \leq j \leq N \ | \ \mathfrak{a}_j \neq 0 \}, \ \text{and} \ M = \sup_{j \in S} |\mathfrak{a}_j |, \ \text{and} \ m =  \inf_{j \in S} |\mathfrak{a}_j | . \]
Suppose that $1, N \in S$, and $|S| = \log N$, and $M \leq \lambda m$. In this case, we can use Corollary $\ref{rp}$ to deduce that
\begin{equation} \label{whyami}
  \int_{[0,1)^{2}} |\sum_{j=1}^{N}  \mathfrak{a}_j e( j \alpha_1 + j^2 \alpha_2) |^{6} d \vec{\alpha} \ll (\log N)^{5/2} |S|^{7/2} M^{6} \leq (\log N)^{6} M^{6}. 
  \end{equation}
Furthermore, we see that the right hand side of $\eqref{BoRes}$ is at least 
\[  e^{ \frac{C\log N}{\log \log N}}|S|^{3} m^{6} \geq  e^{ \frac{C\log N}{\log \log N}} (\log N)^{3} \lambda^{-6} M^{6}, \]
which is much bigger than the right hand side of $\eqref{whyami}$ whenever $N$ is sufficiently large. Thus we note that whenever $\{\mathfrak{a}_{j} \}_{j=1}^{N}$ is a sparse sequence with controlled magnitude, Corollary $\ref{rp}$ provides stronger estimates than $\eqref{BoRes}$.
\par

As we previously mentioned, Theorems $\ref{mainth1}$ and $\ref{mainth2}$ are the culmination of a series of quantitative refinements that we prove throughout this paper. We discuss these results in \S2. We use \S3 to record some preliminary definitions and results which we will frequently use. In \S4, we prove our first wave of non-trivial estimates for $E_{s,2}(A)$ when $s \geq 3$, which builds the base for proofs of Theorems $\ref{mainth1}$ and $\ref{mainth2}$. This will be recorded as Theorem $\ref{main}$, which we will also often describe as the threshold bound. We employ \S5 to derive non-trivial estimates for the $E_{s,3}(A)$. In \S6, we will record a second wave of incidence preliminaries which will set the scene for utilising higher energy techniques. We will derive our first set of estimates for $|s \mathscr{A} - s\mathscr{A}|$ which beat the threshold bounds, in \S7. In \S8, these will then be applied to obtain threshold-breaking bounds for $E_{s,2}(A)$ when $s \geq 4$. We use \S9 to iterate these results for larger values of $s$, and improve our sumset bounds once again. In \S10, we prove a uniform upper bound for the function $r_{s}$. We combine the preceding estimates in \S11 to prove Theorems $\ref{mainth1}$ and $\ref{mainth2}$. We use \S12 to consider more general systems of the form 
\[ \sum_{i=1}^{s} (f(x_i) - f(x_{i+s})) = \sum_{i=1}^{s}( g(x_i) - g(x_{i+s})) = 0, \]
where $f$ and $g$ are polynomials of differing degrees.  We conclude with an appendix where we record the proof of weighted Szemer\'{e}di-Trotter theorem.
\par

\textbf{Acknowledgements.} The author is grateful for support and hospitality from University of Bristol and Purdue University. The author is thankful to Oliver Roche-Newton for helpful conversations. The author would also like to thank Trevor Wooley for his guidance and encouragement. The author is grateful to the anonymous referee for many helpful comments.


\section{Further results}

In this section, we present some of the other results that we prove as we proceed through the paper. From this point till the end of \S10, we will fix $I$ to be an interval on the real line, $\psi : I \to \mathbb{R}$ to be some continuous function and $\mathcal{C} > 0$ to be some parameter such that $I$, $\psi$ and $\mathcal{C}$ satisfy $\eqref{fis1}$ and $\eqref{fis2}$. Moreover, let $s$ be some natural number, and let $A$ be some large, finite, non-empty subset of $I$.  Our aim is to first prove estimates for $E_{s,2}(A)$ and $|s \mathscr{A} - s\mathscr{A}|$ whenever $\psi, I$ and $\mathcal{C}$ are as assumed. We will then use \S11 to deduce Theorems $\ref{mainth1}$ and $\ref{mainth2}$ from these results.
\par

Moreover, it is straightforward to see that for even natural numbers $d \geq 2$, the function $\psi(x) = x^d$ and the interval $I = \mathbb{R}$ satisfy $\eqref{fis1}$ and $\eqref{fis2}$ with $\mathcal{C} \ll_{d} 1$. Similarly, when $d \geq 2$ is an odd number, the function $\psi(x) = x^d$ and the interval $I = (0, \infty)$ satisfy $\eqref{fis1}$ and $\eqref{fis2}$ with $\mathcal{C} \ll_{d} 1$. In \S11, we prove Proposition $\ref{convexity}$ which provides further instances of intervals $I$ and functions $\psi$ that satisfy $\eqref{fis1}$ and $\eqref{fis2}$ with $\mathcal{C} \ll 1$. 
\par

Thus, fixing parameters $I, \psi$ and $\mathcal{C}$, and letting $A \subseteq I$ be some finite, non-empty set and $s$ be some natural number, for each $\vec{n} \in \mathbb{R}^2$, we write $r_{s}(\vec{n})$ to be the number of solutions to $\eqref{twp}$ with $x_i \in A$ for each $1 \leq i \leq s$. Similarly, we define
\[ \mathscr{A} = \{ (a, \psi(a)) \ | \ a \in A \} \ \text{and} \ E_{s,2}(A) = \sum_{\vec{n} \in s \mathscr{A}} r_{s}(\vec{n})^2. \]
\par

Thus we begin with Theorem $\ref{main}$ which provides the first wave of non-trivial upper bounds for $E_{s,2}(A)$. 

\begin{theorem} \label{main}
Let $s \geq 3$. Then we have
\[ E_{s,2}(A) \ll_{\mathcal{C}} |A|^{2s - 3 + 2^{-s+2}} . \]
\end{theorem}

As mentioned in the previous section, this result is sharp up to a factor of $|A|^{2^{-s+2}}$. Moreover, using Cauchy-Schwarz inequality, we see that
\[ |s \mathscr{A} - t \mathscr{A}|E_{s+t,2}(A) \geq |A|^{2(s+t)} \]
for each $s,t \in \mathbb{N}$, and so, we can use Theorem $\ref{main}$ to deduce that 
\begin{equation} \label{anewsoul}
 |s \mathscr{A} - t \mathscr{A}| \gg_{\mathcal{C}} |A|^{3 - 2^{-s-t+2}}, 
 \end{equation}
whenever $s+t \geq 3$.
\par

We can also use Theorem $\ref{main}$ to infer non-trivial estimates for other moments of the function $r_{s}$. Thus, we recall definition $\eqref{fut}$ of $E_{s,3}(A)$
\[ E_{s,3}(A) = \sum_{\vec{n} \in s \mathscr{A}} r_{s}(\vec{n})^3. \]
As with the second moment, we see that $E_{s,3}(A)$ counts the number of solutions to the following system of equations
\begin{align*}
 \sum_{i=1}^{s} \psi(x_i) = \sum_{i=1}^{s} \psi( & x_{i+s}) = \sum_{i=1}^{s} \psi(x_{i+2s})   \\
\sum_{i=1}^{s} x_i = \sum_{i=1}^{s} & x_{i+s} = \sum_{i=1}^{s} x_{i+2s},   
\end{align*}
where $x_i \in A$ for each $1 \leq i \leq 3s$. Our main result in this case is Theorem $\ref{newthird}$, which bounds $E_{s,3}(A)$ in terms of $E_{s-1,2}(A)$. This, when combined with Theorem $\ref{main}$, already provides non-trivial estimates for $E_{s,3}(A)$. 

\begin{theorem} \label{kigi}
Let $s \geq 3$ be a natural number. Then we have
\[ E_{s, 3}(A) \ll_{\mathcal{C}} |A|^{3s - 6 + 2^{-s+3}} \log |A|. \]
\end{theorem}

These third energy estimates appear to be stronger than the corresponding bounds for the second energy given by Theorem $\ref{main}$. In particular, when $s \geq 3$, a straightforward application of Cauchy-Schwarz inequality, combined with estimates provided by Theorem $\ref{kigi}$, recovers Theorem $\ref{main}$ up to factors of $\log |A|$. Thus, when $s \geq 3$, we have
\begin{align*}
E_{s,2}(A) & \leq E_{s,3}(A)^{1/2} |A|^{s/2} \ll |A|^{3s/2 + s/2 - 3 + 2^{-s+2}} (\log |A|)^{1/2} \\
 & = |A|^{2s - 3 +2^{-s+2}}(\log |A|)^{1/2}.
 \end{align*}
We devote \S5 to prove these estimates. 
\par

Our next goal is to strengthen Theorem $\ref{main}$ and Theorem $\ref{kigi}$. In this endeavour, we aim to use higher energy techniques, and so, we use \S6 to prepare some preliminary results in order to use these methods. In \S7, we prove the first set of threshold breaking inequalities for the sumset, that is, we show the following result. 

\begin{theorem} \label{highsum}
We have
\[ |2 \mathscr{A} - 2\mathscr{A}| \gg_{\mathcal{C}} |A|^{3 - 2/11}(\log |A|)^{-18/11} , \]
and
 \[ |3\mathscr{A} - 3 \mathscr{A}| \gg_{\mathcal{C}}  |A|^{3 - 1/23} (\log|A|)^{-36/23}. \]
\end{theorem}

We call these threshold breaking inequalities, because they are stronger than the following estimates
\[ |2 \mathscr{A} - 2\mathscr{A}| \gg_{\mathcal{C}} |A|^{3 - 1/4} \ \text{and} \ |3\mathscr{A} - 3 \mathscr{A}| \gg_{\mathcal{C}}  |A|^{3 - 1/16}, \]
that are provided by $\eqref{anewsoul}$, which in turn is a straightforward corollary of Theorem $\ref{main}$. We further use these sumset estimates to deduce a stronger bound for $E_{s,2}(A)$. Hence in \S8, we prove that
\[ E_{4,2}(A) \ll_{\mathcal{C}}  |A|^{5 + 1/4 - c}, \]
where $c = 1/7246$. We devote \S9 to iterate these bounds for larger values of $s$ in order to obtain estimates that beat Theorem $\ref{main}$ whenever $s \geq 4$. This is recorded as Theorem $\ref{genen}$. These improvements lead to further strengthening of Theorem $\ref{kigi}$, which is written in the form of Theorem $\ref{newthirdnew}$. In particular, we find that
\[E_{s,3}(A) \ll_{\mathcal{C}}  |A|^{3s-6 + (1/4 - c) \cdot 2^{-s + 5}} \log  |A|, \]
whenever $s \geq 5$ and $c = 1/7246.$ We input these bounds in the machinery we develop in \S8, to improve Theorem $\ref{highsum}$.
\par

Finally, we also consider uniform upper bounds for $r_{s}(\vec{n})$ as $\vec{n}$ varies in $s \mathscr{A}$. These can be interpreted as $l^{\infty}$ estimates for the function $r_{s}$ defined for fixed $s, \psi$ and $A$. Thus, given a natural number $s \geq 1$, and a finite set $A$ of real numbers, we define
\[ r_s(A) = \sup_{\vec{n} \in s\mathscr{A}} r_s(\vec{n}). \]
Using the estimates provided by Theorem $\ref{main}$ and Theorem $\ref{genen}$ for $E_{s,2}(A)$, we show the following upper bounds for $r_{s}(A)$ in \S10. 

\begin{theorem} \label{kglw}
Let $A \subseteq I$ be a finite, non-empty set, let $p$ be a natural number and let $c= 1/7246$. When $p \geq 2$, we have
\[ r_{2p+1}(A) \ll_{\mathcal{C}} |A|^{2p-2 + \sigma_{p}}, \]
where $\sigma_{2} = 2/3$ and $\sigma_{3} = 1/3$, and $\sigma_{p} =  (1/4 - c)\cdot 2^{-p+5} \cdot 3^{-1}$ when $p \geq 4$. Similarly, when $p \geq 3$, we have
\[ r_{2p}(A) \ll_{\mathcal{C}} |A|^{2p-3 + \sigma_{p}}, \]
where $\sigma_{3} = 1/2$ and $\sigma_{4} = 1/4 - c/3$, and $\sigma_{p} = (1/4 - c) \cdot 2^{-p + 4}$ when $p \geq 5$. 
\end{theorem}

We note that when $A$ is a set of integers and $p \geq 3$ is an odd number, Theorem $\ref{kglw}$ gives stronger estimates for $r_{2p+1}(A)$, than what one could infer from a combination of Theorem $\ref{mainth1}$ and a convexity estimate of the form
\[ r_{2p+1}(A) \leq E_{p,2}(A)^{1/2} E_{p+1,2}(A)^{1/2}. \] 
For instance, the right hand side above may be bounded by $O(|A|^{2p - 2 + (1/4 - c) \cdot 2^{-p+2} \cdot 3})$ in the case when $p \geq 4$ using Theorem $\ref{mainth1}$, whereupon, one may verify that Theorem $\ref{kglw}$ provides a much stronger conclusion.  
\par

We conclude this section by presenting some more instances of functions $\psi$, intervals $I$ and constants $\mathcal{C}$ that satisfy $\eqref{fis1}$ and $\eqref{fis2}$. Thus, let $I = (0, \infty)$. For this choice of $I$, we can use Proposition $\ref{convexity}$ to see that the function $\psi_1(x) = \log x$, the function $\psi_2(x) = x^{\lambda}$ for some fixed $\lambda \in (1, \infty)$, and the function $\psi_3(x) = x^{-1}$ satisfy $\eqref{fis1}$ and $\eqref{fis2}$. In all of these cases, we can choose $\mathcal{C} = O(1)$. 
\par

In particular, setting $\psi(x) = \log x$, and letting $A$ be some finite, non-empty subset of $(0, \infty)$, we note that $E_{s,2}(A)$ studies solutions to systems of equations of the form
\[ x_1 + \dots + x_{s} = y_{1} + \dots + y_{s} \ \text{and} \ x_1 x_2 \dots x_s = y_1 y_2 \dots y_s, \]
where $x_i,y_i \in A$ for each $1 \leq i \leq s$. Similarly, we have
\[ | s \mathscr{A}| = |\{ (a_1 + \dots + a_s , a_1 a_2 \dots a_s) \ | \ a_1, \dots, a_s \in A \} |. \]
Thus, estimates for $E_{s,2}(A)$ and $|s \mathscr{A}|$, when $\psi(x) = \log x$, are closely related to the sum-product phenomenon, where one simultaneously studies additive and multiplicative properties of an arbitrary set of real numbers.


\section{Preliminaries I}

We begin by introducing some notation. In this paper, we use Vinogradov notation, that is, we write $X \gg Y$, or equivalently $Y \ll X$, to mean $|X| \geq C |Y|$ where $C$ is some positive constant. The corresponding notation $X \gg_{z} Y$ is defined similarly, except in this case the constant $C$ will depend on the parameter $z$. Next, given a non-empty, finite set $Z$, we use $|Z|$ to denote the cardinality of $Z$. Moreover, for every natural number $k \geq 2$, we will use boldface to denote vectors $\vec{x} = (x_1, x_2, \dots, x_k) \in \mathbb{R}^k$. 
\par

We now introduce some definitions from additive combinatorics. Given two finite, non-empty sets $X,Y$ of an abelian group $G$, we define the additive energy
\[ E(X,Y) = |\{ (x_1, x_2, y_1, y_2) \in X^2 \times Y^2  \ | \ x_1 + y_1 = x_2 + y_2 \}|. \]
We can also define the sumset $X+Y$ and the difference set $X-Y$ as 
\[ X+Y = \{ x+ y \ | \ (x,y) \in X \times Y \} \ \text{and} \  X-Y = \{ x- y \ | \ (x,y) \in X \times Y \}  .\]
As in \S1, we can use Cauchy-Schwarz inequality to see that 
\[ E(X,Y) |X+Y| \geq |X|^2 |Y|^2 \ \text{and} \  E(X,Y) |X-Y| \geq |X|^2 |Y|^2. \]
Moreover for a natural number $k \geq 2$, we recall the definition of the $k$-fold sumset
\[ k X = \{ x_1 + \dots + x_k \ | \ x_i \in X \ \text{for each} \ 1 \leq i \leq k \}. \]
\par

One of the primary tools that we will use throughout this paper is a weighted version of the Szemer\'{e}di-Trotter theorem. In order to state this, we first record some notation. Let $\mathcal{D} >0$ be a parameter. A finite, non-empty collection $L$ of curves in $\mathbb{R}^2$ is called \emph{$\mathcal{D}$-valid} if each pair of curves intersect at no more than $O_{\mathcal{D}}(1)$ points, and at most $O_{\mathcal{D}}(1)$ number of curves pass through any pair of points. 
\par

Furthermore, for a set $L$ of curves, let $w : L \to \mathbb{N}$ be a weight function on $L$. Given such a weight function, we denote 
\[ \n{L}_{\infty} = \sup_{l \in L} |w(l)|, \ \text{and} \ \n{L}_{i}^{i} = \sum_{l \in L} |w(l)|^i \ \text{for each} \ i \in \{1,2\}. \]

These are analogues of the $l^1$, $l^2$ and $l^{\infty}$ norm defined for the sequence $\{ w(l) \}_{l \in L}$. Moreover, let $P$ be a finite, non-empty set of points in $\mathbb{R}^2$. We can also define a weight function $w' : P \to \mathbb{N}$ on $P$, and as before, define the norms $\n{P}_1$, $\n{P}_2$ and $\n{P}_{\infty}$. Given a set of points $P$ and a set of curves $L$ with their associated weight functions $w$ and $w'$, we define the number of weighted incidences $I_{w,w'}(P,L)$ as 
\[ I_{w,w'}(P,L) = \sum_{p \in P} \sum_{l \in L} \mathds{1}_{p \in l} w'(p)w(l). \]
With this notation in hand, we now state the weighted Szemer\'{e}di-Trotter theorem.

\begin{lemma} \label{wtst1}
Let $P$ be a set of points in $\mathbb{R}^2$, let $L$ be a set of $\mathcal{D}$-valid curves in $\mathbb{R}^2$ and let $w, w'$ be weight functions on $L$ and $P$ respectively. Then we have
\begin{equation*} 
 I_{w,w'}(P,L) \ll_{\mathcal{D}}  (\n{P}_1 \n{P}_2^2\n{L}_1 \n{L}_2^2)^{1/3} + \n{L}_{\infty}\n{P}_1  +  \n{P}_{\infty}\n{L}_1. 
 \end{equation*}
\end{lemma}

Lemma $\ref{wtst1}$ follows from a dyadic decomposition argument combined with the unweighted version of Szemer\'{e}di-Trotter theorem for points and $\mathcal{D}$-valid curves in $\mathbb{R}^2$ \cite[Theorem 1.1]{PS1998}. We state the proof of this result at the end of the paper.
\par

In particular, we note a special case of the above, when the point set $P$ is unweighted, that is, when the weight function $w'$ satisfies
\[ w'(p) = 1 \ \text{for all} \ p \in P.\]
In this case, we write $ I_{w,w'}(P,L) = I_{w}(P,L)$, and so, Lemma $\ref{wtst1}$ furnishes the following corollary.

\begin{Corollary} \label{wtst}
Let $P$ be a set of points in $\mathbb{R}^2$, let $L$ be a set of $\mathcal{D}$-valid curves in $\mathbb{R}^2$ and let $w$ be a weight function on $L$. Then we have
\begin{equation*} 
 I_{w}(P,L) \ll_{D}  |P|^{2/3}(\n{L}_1 \n{L}_2^2)^{1/3} + \n{L}_{\infty}|P|  +  \n{L}_1. 
 \end{equation*}
\end{Corollary}

For each $t \in \mathbb{N}$, we define $P_{t} \subseteq P$ as 
\[ P_{t} = \{ p \in P \ | \ \sum_{l \in L} \mathds{1}_{p \in l} w(l) \geq t \}. \]
In incidence theory, $P_t$ is known as the set of $t$-rich points. We note that for each $t \in \mathbb{N}$, we have
\[ t|P_t| \leq I_{w}(P,L). \]
Combining this with Corollary $\ref{wtst}$, we obtain
\begin{equation} \label{ches}
t|P_t| \ll_{D} |P_t|^{2/3} (\n{L}_1 \n{L}_2^2)^{1/3} + \n{L}_{\infty} |P_t| + \n{L}_1.
\end{equation}
\par

Throughout this paper, we intend to apply the weighted Szemer\'{e}di-Trotter theorem to obtain estimates on various moments of the function $r_{s}$, for a fixed set $A$. For our purposes, the set of curves $L$ on which we apply the weighted Szemer\'{e}di-Trotter theorem will always be a collection of translates of the curve $y = \psi (x)$. Thus, our next goal is to show that whenever an interval $I$, a continuous function $\psi : I \to \mathbb{R}$ and a constant $\mathcal{C} > 0$ satisfy $\eqref{fis1}$ and $\eqref{fis2}$, then any finite collection $L$ of translates of the curves $y = \psi(x)$ is a set of $\mathcal{C}$-valid curves. This is recorded in the following proposition.

\begin{Proposition} \label{aslo}
Let $I$ be an interval on the real line, let $\psi : I \to \mathbb{R}$ be a continuous function, and let $\mathcal{C} > 0$ be a constant such that $I$, $\psi$ and $\mathcal{C}$ satisfy $\eqref{fis1}$ and $\eqref{fis2}$. Moreover, let $X$ be some finite, non-empty set of points in $\mathbb{R}^2$ such that $|X| \geq 2$, and let $A \subseteq I$ be a finite, non-empty set. Then, the set of curves $L$, defined as
\[ L = \{ l_{\vec{x}} \ | \ \vec{x} \in X \} \ \text{where} \ l_{\vec{x}} = \{ (t, \psi(t)) \ | \ t \in I \} + \vec{x}, \]
 is a set of $\mathcal{C}$-valid curves. Furthermore, we have
 \begin{equation} \label{r11}
  \sup_{\vec{n} \in s \mathscr{A}} r_{s}(\vec{n}) \ll_{\mathcal{C}} |A|^{s-2} \ \text{and} \ E_{s,2}(A) \ll_{\mathcal{C}} |A|^{2s-2}, 
  \end{equation}
 for each natural number $s \geq 2$. 
\end{Proposition}

\begin{proof}
Let $l_{\vec{u}}$ and $l_{\vec{v}}$ be two fixed, distinct curves in $L$. Let $(n_1, n_2)$ be a point lying on both $l_{\vec{u}}$ and $l_{\vec{v}}$. Thus there exist $x, y \in I$ such that
\[ n_1 = x + u_1 = y + v_1 \ \text{and} \ n_2 = \psi(x) + u_2 = \psi(y) + v_2 . \]
Thus, we see that
\begin{equation} \label{revamp1}
x - y - (v_1 - u_1) = \psi(x) - \psi(y) - (v_2 - u_2) = 0.
 \end{equation}
Moreover we note that $u_1 \neq v_1$, since otherwise, we would have $x= y$, which would then imply that $u_2 = v_2$, thus contradicting our assumption that $l_{\vec{u}}$ and $l_{\vec{v}}$ are two distinct curves. Setting $\delta_1 = v_1 - u_1$ and $\delta_2 = v_2 - u_2$, we use $\eqref{fis1}$ to deduce that there are at most $O_{\mathcal{C}}(1)$ solutions to $\eqref{revamp1}$. Thus, there are at most $O_{\mathcal{C}}(1)$ possibilities for the choice of $x$. Each such $x$ fixes $(n_1, n_2)$, whence, there are at most $O_{\mathcal{C}}(1)$ possible values of $(n_1, n_2)$. Consequently, we infer that for fixed $\vec{u}$ and $\vec{v}$, there are at most $O_{\mathcal{C}}(1)$ points of intersection between $l_{\vec{u}}$ and $l_{\vec{v}}$.
\par

Similarly, let $(x_1, x_2)$ and $(y_1, y_2)$ be two fixed, distinct points in $\mathbb{R}^2$, and let $l_{\vec{u}}$ be a curve in $L$ passing through both the points. Thus, there exist $m,n \in I$ such that
\[ x_1 = m + u_1 \ \text{and} \ x_2 = \psi(m) + u_2, \]
and 
\[ y_1 = n + u_1 \ \text{and} \ y_2 = \psi(n) + u_2. \]
From these expressions, we see that
\begin{equation} \label{revamp2}
\psi(m) - \psi(n) -(x_2 - y_2) = 0 \ \text{and} \ m - n - ( x_1 - y_1) = 0. 
 \end{equation}
As before, we note that $x_1 \neq y_1$, since otherwise, we would have $m = n$ which would then imply that $x_2 = y_2$, consequently contradicting our assumption that $(x_1, x_2)$ and $(y_1, y_2)$ are two distinct points. Thus, we can use $\eqref{fis1}$ to infer that there are $O_{\mathcal{C}}(1)$ values of $m$ that satisfy $\eqref{revamp2}$. Since each such value of $m$ fixes $(u_1, u_2)$, we deduce that there are at most $O_{\mathcal{C}}(1)$ curves in the collection $L$ that pass through both $(x_1, x_2)$ and $(y_1, y_2)$. This implies that our set of curves $L$ is a $\mathcal{C}$-valid collection. 
\par

We note that 
\[ r_{s}(A) \leq |A|^{s-2} r_{2}(A) \]
for each natural number $s \geq 2$. Moreover, we see that $\eqref{fis2}$ implies that
\[ r_{2}(A) = \sup_{\vec{n} \in 2 \mathscr{A}} |\{ (x,y) \in A^2 \ | \  x+y = n_1 \ \text{and} \ \psi(x) + \psi(y) = n_2 \}| \ll_{\mathcal{C}}1 , \]
whence, 
\[ r_{s}(A) \ll_{\mathcal{C}} |A|^{s-2}. \]
Combining this with the the fact that
\[ E_{s,2}(A) = \sum_{\vec{n} \in s\mathscr{A}} r_{s}(\vec{n})^2 \leq r_{s}(A) \sum_{\vec{n} \in s\mathscr{A}} r_{s}(\vec{n}) = |A|^{s} r_{s}(A), \]
we see that
\[ E_{s,2}(A) \ll_{\mathcal{C}} |A|^{2s-2}.\]
Thus we conclude the proof of Proposition $\ref{aslo}$.
\end{proof}


\section{Proof of Theorem $\ref{main}$}

In this section, we proceed with the proof of Theorem $\ref{main}$. In particular, we will derive Theorem $\ref{main}$ as a consequence of stronger result that gives non-trivial estimates for $E_{s,2}(A)$ in terms of $E_{s-1,2}(A)$. We record this result below. 

\begin{theorem} \label{thiswillpass}
Let $s \geq 3$. Then for every finite, non-empty set $A \subseteq I$, we have
\[ E_{s,2}(A) \ll_{\mathcal{C}} |A|^{2s-3} + |A|^{s-1/2}E_{s-1,2}(A)^{1/2}. \]
\end{theorem}

\begin{proof}[Proof of Theorem $\ref{thiswillpass}$]

For each $\vec{u} \in (s-1)\mathscr{A}$, we define the curve
\[l_{\vec{u}} = \{(t,\psi(t)) \ | \ t \in \mathbb{R} \} + \vec{u}  \]
and its associated weight
\[ w(\vec{u}) = r_{s-1}(\vec{u}). \]
Further, our set of curves $L$ will be
\[ L = \{ l_{\vec{u}} \ | \ \vec{u} \in (s-1) \mathscr{A} \}. \]
Similarly, for each $\vec{v} \in s \mathscr{A}$, we define the point $p_{\vec{v}} = \vec{v}$ and the associated weight $w'(\vec{v}) = r_{s}(\vec{v})$. Our set of points $P$ will be
\[ P = \{ p_{\vec{v}} \ | \ \vec{v} \in s \mathscr{A} \}. \]
\par

We note that
\[ E_{s,2}(A) =  \sum_{\vec{a}_1, \dots, \vec{a}_{2s} \in \mathscr{A}} \mathds{1}_{\sum_{i=1}^{s} \vec{a}_i = \sum_{i=1}^{s} \vec{a}_{i+s}} = \sum_{\vec{a} \in \mathscr{A} } \sum_{\vec{u} \in s \mathscr{A}} \sum_{\vec{v} \in (s-1) \mathscr{A}}  r_{s-1}(\vec{v}) r_{s}(\vec{u}) \mathds{1}_{ \vec{a} + \vec{v} = \vec{u}}.\]
Moreover, we can write
\[ \sum_{\vec{a} \in \mathscr{A} } \sum_{\vec{u} \in s \mathscr{A}} \sum_{\vec{v} \in (s-1) \mathscr{A}}  r_{s-1}(\vec{v}) r_{s}(\vec{u}) \mathds{1}_{ \vec{a} + \vec{v} = \vec{u}}  \leq  \sum_{\vec{u} \in s \mathscr{A}} \sum_{\vec{v} \in (s-1) \mathscr{A}}  w(\vec{v}) w'(\vec{u}) \mathds{1}_{ p_{\vec{u}} \in l_{\vec{v}}} . \] 
From the preceding expressions, we discern that
\[ E_{s,2}(A) \leq I_{w,w'}(P,L) . \]
Furthermore, using Proposition $\ref{aslo}$, we see that $L$ is a collection of $\mathcal{C}$-valid curves, and so we can use Lemma $\ref{wtst1}$ to obtain
\[ I_{w,w'}(P,L) \ll_{\mathcal{C}} (\n{P}_1 \n{P}_2^2\n{L}_1 \n{L}_2^2)^{1/3} + \n{L}_{\infty}\n{P}_1  +  \n{P}_{\infty}\n{L}_1. \]
We note that
\[ \n{P}_1 = \sum_{\vec{v} \in s \mathscr{A}} r_{s}(\vec{v}) = |A|^{s}, \ \text{and} \ \n{P}_2^2 = \sum_{\vec{v} \in s \mathscr{A}} r_{s}(\vec{v})^2 = E_{s,2}(A). \]
Similarly, we have
\[ \n{L}_1 = |A|^{s-1}, \ \text{and} \ \n{L}_2^2 = E_{s-1,2}(A). \]
Lastly, we see that
\[ \n{P}_{\infty} = \sup_{\vec{v} \in s\mathscr{A}} r_{s}(\vec{v}) \ll_{\mathcal{C}} |A|^{s-2}, \ \text{and} \ \n{L}_{\infty} = \sup_{\vec{v} \in (s-1)\mathscr{A}} r_{s-1}(\vec{v}) \ll_{\mathcal{C}} |A|^{s-3}.\]
Combining these estimates together, we find that
\[ E_{s,2}(A) \ll_{\mathcal{C}} |A|^{(2s-1)/3}E_{s,2}(A)^{1/3} E_{s-1,2}(A)^{1/3} +  |A|^{2s-3}. \]
Thus, we have either
\[ E_{s,2}(A) \ll_{\mathcal{C}} |A|^{2s - 3}  \]
or
\[ E_{s,2}(A) \ll_{\mathcal{C}} |A|^{s-1/2} E_{s-1,2}(A)^{1/2}. \]
In either case, we see that Theorem $\ref{thiswillpass}$ holds true. 
\end{proof}

As previously mentioned, we can derive Theorem $\ref{main}$ as a straightforward consequence of Theorem $\ref{thiswillpass}$ and the fact that $E_{2,2}(A) \ll_{\mathcal{C}} |A|^2$. In particular, when $s=3$, we use Theorem $\ref{thiswillpass}$ along with the fact that $E_{2,2}(A) \ll_{\mathcal{C}} |A|^2$ to obtain
\[ E_{3,2}(A) \ll_{\mathcal{C}} |A|^{3-1/2}|A| = |A|^{3 + 1/2}. \]
Henceforth, Theorem $\ref{main}$ holds when $s=3$. Given $s > 3$, if we assume that our theorem holds for $s-1$, that is, if we assume that
\[ E_{s-1,2}(A) \ll_{\mathcal{C}} |A|^{2s-2 - 3 + 2^{-s+3}} \]
is true, then noting Theorem $\ref{thiswillpass}$, we deduce that
\[ E_{s,2}(A) \ll_{\mathcal{C}} |A|^{s-1/2 + s - 1 - 3/2 + 2^{-s+2}} = |A|^{2s -3 +  2^{-s+2}}. \]
Consequently, we finish the proof of Theorem $\ref{main}$ by induction.
\par


\section{Third energy estimates}

We utilise this section to prove two third-energy estimates, the first of these being Theorem $\ref{kigi}$, and the second concerning a variant of $E_{2s,3}(A)$.  
\par

Let $s \geq 3$ be a natural number. Our aim is to first bound $E_{s,3}(A)$ in terms of $E_{s-1,2}(A)$ and then combine this estimate with Theorem $\ref{main}$ to deliver Theorem $\ref{kigi}$. Thus, we prove the following theorem. 

\begin{theorem} \label{newthird}
Let $s \geq 3$. Then we have
\[ E_{s,3}(A) \ll_{\mathcal{C}} |A|^{s-1} E_{s-1,2}(A) \log |A| + |A|^{3s-6}. \]
\end{theorem}

\begin{proof}[Proof of Theorem $\ref{newthird}$]
 We begin by using $\eqref{fut}$ to dyadically decompose $E_{s,3}(A)$ as 
\[ E_{s, 3}(A)  = \sum_{j=0}^{J} \sum_{\vec{n} \in P_{2^{j}}} r_{s}(\vec{n})^3, \]
where  $J$ is the largest natural number such that
\[ 2^{J} \leq \sup_{\vec{n} \in s\mathscr{A}} r_{s}(\vec{n}) , \]
and for each $0 \leq j \leq J$, we write
\[ P_{2^{j}}=  \{ \vec{n} \in s\mathscr{A} \ | \ 2^{j} \leq  r_{s}(\vec{n}) < 2^{j+1} \}. \]
\par

As in the previous section, for each $\vec{u} \in (s-1)\mathscr{A}$, we define the curve
\[l_{\vec{u}} = \{(t,\psi(t)) \ | \ t \in \mathbb{R} \} + \vec{u}  \]
and its associated weight
\[ w(\vec{u}) = r_{s-1}(\vec{u}). \]
Moreover, we write $L$ to be
\[ L = \{ l_{\vec{u}} \ | \ \vec{u} \in (s-1) \mathscr{A} \}. \]
As before, we can use Proposition $\ref{aslo}$ to see that $L$ is a $\mathcal{C}$-valid collection of curves. We note that for each $\vec{n} \in s \mathscr{A}$, we have
\begin{equation} \label{iri}
 r_{s}(\vec{n}) \leq \sum_{l_{\vec{u}} \in L} \mathds{1}_{\vec{n} \in l_{\vec{u}}} r_{s-1}(\vec{u}) = \sum_{l_{\vec{u}} \in L} \mathds{1}_{\vec{n} \in l_{\vec{u}}} w(l_{\vec{u}}). 
\end{equation}
Thus, for each $0 \leq j \leq J$ and for each $\vec{n} \in P_{2^j}$, we deduce that $\sum_{l_{\vec{u}} \in L} \mathds{1}_{\vec{n} \in l_{\vec{u}}} w(l_{\vec{u}}) \geq 2^j$, 
and so, as in $\eqref{ches}$, we can use Corollary $\ref{wtst}$ to obtain the bound
\begin{equation} \label{rev1}
 2^{j}|P_{2^{j}}|  \ll_{\mathcal{C}} |P_{2^{j}} |^{2/3} |A|^{(s-1)/3} E_{s-1,2}(A)^{1/3} + r_{s-1}(A) |P_{2^{j}}| +  |A|^{s-1} . 
 \end{equation}
 \par
 
We define
\[ U = \{ 0 \leq j \leq J \ | \ 2^{j} \leq  \mathcal{D'} r_{s-1}(A) \}, \]
and
\[ V =  \{0,1,\dots, J\} \setminus U = \{ 0 \leq j \leq J \ | \ 2^{j} >  \mathcal{D'} r_{s-1}(A) \}, \]
where $\mathcal{D'}$ is a sufficiently large constant depending on $\mathcal{C}$. We note that
\begin{align*} \sum_{j \in U} \sum_{\vec{n} \in P_{2^{j}}} r_{s}(\vec{n})^3 
& \leq \mathcal{D'}^2 r_{s-1}(A)^2 \sum_{j \in U} \sum_{\vec{n} \in P_{2^{j}}} r_{s}(\vec{n}) \\
& \ll_{\mathcal{C}} |A|^{2s-6} \sum_{\vec{n} \in s \mathscr{A}} r_{s}(\vec{n}) = |A|^{3s-6},
\end{align*}
where we have used $\eqref{r11}$ to deduce the last inequality. 
Thus, we have
\begin{align}  \label{weonly}
   E_{s, 3}(A)  & =  \sum_{j \in U}  \sum_{\vec{n} \in P_{2^{j}}} r_{s}(\vec{n})^3 +  \sum_{j \in V} \sum_{\vec{n} \in P_{2^{j} } } r_{s}(\vec{n})^3 \nonumber \\
& \ll_{\mathcal{C}}  |A|^{3s-6} + \sum_{j \in V} \sum_{\vec{n} \in P_{2^{j} } } 2^{3j},
\end{align}
and consequently, it suffices to consider the case when $j \in V$. 
\par

Since $\mathcal{D'}$ is a sufficiently large constant, we can use $\eqref{rev1}$ to deduce that 
\[ 2^{j}|P_{2^{j}}|  \ll_{\mathcal{C}} |P_{2^{j}} |^{2/3} |A|^{(s-1)/3} E_{s-1,2}(A)^{1/3} +  |A|^{s-1} , \]
whenever $j \in V$. This implies that 
\begin{equation} \label{lopa}
  |P_{2^{j}}|  \ll_{\mathcal{C}}  |A|^{s-1} E_{s-1,2}(A) 2^{-3j} + |A|^{s-1} 2^{-j}, \end{equation}
for all $j \in V$.  Thus, using $\eqref{lopa}$, we discern the existence of some constant $K_1 > 0$, depending on $\mathcal{C}$, such that we have
\begin{equation} \label{tut1}
 |P_{2^{j}}| \ll_{\mathcal{C}}  |A|^{s-1} E_{s-1,2}(A) 2^{-3j} 
 \end{equation}
whenever $2^{j} \leq K_1 E_{s-1,2}(A)^{1/2}$, and
\begin{equation} \label{tut2}
 |P_{2^{j}}| \ll_{\mathcal{C}}  |A|^{s-1} 2^{-j} 
 \end{equation}
whenever $2^{j} > K_1 E_{s-1,2}(A)^{1/2}$.
\par

We now prove another estimate for $|P_{\tau}|$ which improves upon $\eqref{lopa}$ whenever $\tau$ is exceptionally large. In this endeavour, we will follow an idea from \cite{Wo2002b} and prove the following lemma.

\begin{lemma} \label{slim}
There exists some large constant $K_2 > 0$, depending on $\mathcal{C}$, such that
\begin{equation*}
 |P_{\tau}| \ll_{\mathcal{C}} E_{s-1,2}(A) |A| \tau^{-2}  ,
 \end{equation*}
whenever $\tau \geq K_2 E_{s-1,2}(A)^{1/2}$. 
\end{lemma}

\begin{proof}[Proof of Lemma $\ref{slim}$]
We begin by noting that
\[ \tau |P_{\tau}| \leq \sum_{\vec{n} \in P_{\tau}} r_{s}(\vec{n}) = \sum_{ \substack{ \vec{a}_1, \dots, \vec{a}_s \in \mathscr{A} \\ \vec{n} \in P_{\tau} } } \mathds{1}_{ \vec{n} = \vec{a}_1 + \dots + \vec{a}_s}. \]
We can rewrite the above as 
\[ \sum_{ \substack{ \vec{a}_1, \dots, \vec{a}_s \in \mathscr{A} \\ \vec{n} \in P_{\tau} } } \mathds{1}_{ \vec{n} = \vec{a}_1 + \dots + \vec{a}_s} = \sum_{\vec{u} \in \mathbb{R}^2} \sum_{ \substack{ \vec{a}_1 \in \mathscr{A} \\ \vec{n} \in P_{\tau} } } \mathds{1}_{ \vec{n} - \vec{a}_1 = \vec{u}} \sum_{\vec{a_{2}}, \dots, \vec{a_{s}} \in \mathscr{A}}  \mathds{1}_{\vec{u} =  \vec{a}_2 + \dots + \vec{a}_s}. \]
We now apply Cauchy-Schwarz inequality to obtain the following upper bound for the expression above
\[  \bigg( \sum_{\vec{u} \in \mathbb{R}^2} \bigg(  \sum_{ \substack{ \vec{a}_1 \in \mathscr{A} \\ \vec{n} \in P_{\tau} } } \mathds{1}_{ \vec{n} - \vec{a}_1 = \vec{u}} \bigg)^2 \ \bigg)^{1/2}   \bigg( \sum_{\vec{u} \in \mathbb{R}^2} \bigg(   \sum_{\vec{a_{2}}, \dots, \vec{a_{s}} \in \mathscr{A}}  \mathds{1}_{\vec{u} =  \vec{a}_2 + \dots + \vec{a}_s} \bigg)^{2} \ \bigg)^{1/2}. \]
We note that
\[ \sum_{\vec{u} \in \mathbb{R}^2} \bigg(   \sum_{\vec{a_{2}}, \dots, \vec{a_{s}} \in \mathscr{A}}  \mathds{1}_{\vec{u} =  \vec{a}_2 + \dots + \vec{a}_s} \bigg)^{2} = E_{s-1,2}(A) , \]
and 
\[  \sum_{\vec{u} \in \mathbb{R}^2} \bigg(  \sum_{ \substack{ \vec{a}_1 \in \mathscr{A} \\ \vec{n} \in P_{\tau} } } \mathds{1}_{ \vec{n} - \vec{a}_1 = \vec{u}} \bigg)^2 = E(P_{\tau}, \mathscr{A}), \]
where $E(P_{\tau}, \mathscr{A})$ counts the number of solutions to 
\begin{equation} \label{tnt}
 \vec{a}_1 - \vec{a}_2 = \vec{n} - \vec{m},
 \end{equation}
with $\vec{a}_1, \vec{a}_2 \in \mathscr{A}$ and $\vec{n}, \vec{m} \in P_{\tau}$. Combining the above inequalities, we find that
\begin{equation} \label{inof}
 \tau |P_{\tau}| \leq E(P_{\tau}, \mathscr{A})^{1/2} E_{s-1,2}(A)^{1/2}. 
 \end{equation}
\par

We now record a fairly straightforward upper bound for $ E(P_{\tau}, \mathscr{A})$. Let $\vec{a}_1, \vec{a}_2 \in \mathscr{A}$ and $\vec{n}, \vec{m} \in P_{\tau}$ satisfy $\eqref{tnt}$. If $\vec{n} = \vec{m}$, then $\vec{a}_1 = \vec{a}_2$, thus contributing to $|A||P_{\tau}|$ solutions to $\eqref{tnt}$. In the case $\vec{n} \neq \vec{m}$, we have $|P_{\tau}|^2$ choices for $\vec{n}, \vec{m}$, and for each such fixed choice, $\eqref{fis1}$ implies that there are $O_{\mathcal{C}}(1)$ number of pairs $\vec{a}_1, \vec{a}_2$ that satisfy $\eqref{tnt}$. Consequently, we have 
\[ E(P_{\tau}, \mathscr{A}) \ll_{\mathcal{C}} |A||P_{\tau}| + |P_{\tau}|^2. \]
Combining this with $\eqref{inof}$, we get
\[ \tau |P_{\tau}| \ll_{\mathcal{C}} E_{s-1,2}(A)^{1/2} (|A|^{1/2}|P_{\tau}|^{1/2} + |P_{\tau}|). \]
Thus there exists some large constant $K_2 > 0$, only depending on $\mathcal{C}$, such that
\[  |P_{\tau}| \ll_{\mathcal{C}} E_{s-1,2}(A) |A| \tau^{-2} , \]
whenever $\tau \geq K_2 E_{s-1,2}(A)^{1/2}$. This concludes the proof of our lemma.
\end{proof}

We now have enough tools to tackle the case when $j \in V$. We begin by writing $V$ as a union of sets 
\begin{equation} \label{superf} 
 V = U_1 \cup U_2 \cup U_3, 
\end{equation}
where
\begin{align*}
U_1 = & \{ 0 \leq j \leq J \ | \  \mathcal{D'} r_{s-1}(A) < 2^{j} \leq K_1 E_{s-1,2}(A)^{1/2}  \},\\
U_2 = & \{ 0 \leq j \leq J \ | \  K_1 E_{s-1,2}(A)^{1/2} < 2^{j} \leq K_2 E_{s-1,2}(A)^{1/2}  \} \ \text{and} \\
U_3 = & \{ 0 \leq j \leq J \ | \   K_2 E_{s-1,2}(A)^{1/2} < 2^{j} \}.  
\end{align*}
We remark that depending on the values of $K_1$ and $K_2$, the sets $U_1$ and $U_3$ could have a non-empty intersection, and the set $U_2$ could be empty, but this does not affect our proof.
\par

Noting $\eqref{weonly}$ and $\eqref{superf}$, it suffices to prove that
\begin{equation} \label{rtb}
 \sum_{j \in U_i} |P_{2^{j}}|  2^{3j} \ll_{\mathcal{C}} |A|^{s-1} E_{s-1,2}(A) \log |A|, 
 \end{equation}
for each $1 \leq i \leq 3$. We first consider the case when $j \in U_1$. For these values of $j$, we can use $\eqref{tut1}$ to infer that
\[ |P_{2^{j}}| \ll_{\mathcal{C}}    |A|^{s-1} E_{s-1,2}(A) 2^{-3j} .\]
Thus, we have
\[ \sum_{j \in U_1} |P_{2^{j}}|  2^{3j} \ll_{\mathcal{C}}|A|^{s-1} E_{s-1,2}(A) \sum_{j \in U_1} 1 \ll_{\mathcal{C}}|A|^{s-1} E_{s-1,2}(A) \log |A|, \]
in which case, we are done.
\par

Our next goal is to show that 
\[ \sum_{j \in U_2} |P_{2^{j}}| 2^{3j} \ll_{\mathcal{C}} |A|^{s-1} E_{s-1,2}(A)  . \]
If $K_1 \geq K_2$, then $U_2$ is empty, in which case we are done. Thus we can assume that $K_2 > K_1$, whereupon, we use $\eqref{tut2}$ to deduce that
\[ |P_{2^{j}}| \ll_{\mathcal{C}} |A|^{s-1} 2^{-j} \ \text{and} \  2^{j} \leq K_2 E_{s-1,2}(A)^{1/2}  . \]
Combining the above inequalities along with the fact that $|U_2| \ll \log(K_2/K_1) \ll_{\mathcal{C}} 1$, we have that
\[ \sum_{j \in U_2} |P_{2^{j}}| 2^{3j} \ll_{\mathcal{C}} |A|^{s-1}  \sum_{j \in U_2} 2^{2j} \ll_{\mathcal{C}} |A|^{s-1} E_{s-1,2}(A). \]
Thus when $j \in U_2$, we are done. 
\par

We finally consider the case when $j \in U_3$. Noting Lemma $\ref{slim}$, we see that 
\[ |P_{2^{j}}| \ll_{\mathcal{C}}  E_{s-1,2}(A) |A|  2^{-2j} \]
whenever $j \in U_3$. Thus, we have the desired bound
\begin{align*}
 \sum_{j \in U_3} |P_{2^{j}}| 2^{3j} & \ \ll_{\mathcal{C}} \ E_{s-1,2}(A)|A| \sum_{j \in U_3} 2^{j} \ \ll_{\mathcal{C}} \ E_{s-1,2}(A)|A| 2^{J} \\
 &  \ \ll_{\mathcal{C}} \ E_{s-1,2}(A)|A| r_{s}(A) \ \ll_{\mathcal{C}}E_{s-1,2}(A)|A|^{s-1}, 
 \end{align*}
where the last inequality follows from $\eqref{r11}$. With this, we conclude our proof of $\eqref{rtb}$, and consequently, Theorem $\ref{newthird}$. 
\end{proof}

We note that Theorem $\ref{kigi}$ follows as a straightforward application of Theorem $\ref{main}$ and Theorem $\ref{newthird}$. In particular, the latter implies that
\[ E_{s,3}(A) \ll_{\mathcal{C}} |A|^{s-1} E_{s-1,2}(A) \log |A| + |A|^{3s-6}, \]
which when combined with estimates from Theorem $\ref{main}$ for $E_{s-1,2}(A)$, gives us
\[ E_{s,3}(A) \ll_{\mathcal{C}} |A|^{s-1} |A|^{2s - 5 + 2^{-s + 3}} \log |A| + |A|^{3s-6}\ll |A|^{3s-6 + 2^{-s+3} } \log |A|. \]
Thus Theorem $\ref{kigi}$ holds true.
\par

Using these techniques, we can also get upper bounds for a variant of $E_{s,3}(A)$. We let $s \geq 1$. We begin by defining a variant of $r_{2s}(\vec{n})$, that is, we write
\begin{equation} \label{defpop}
 r'_{2s}(\vec{n}) = |\{ (\vec{a}_1, \dots, \vec{a}_{2s}) \in \mathscr{A}^{2s} \ | \ \vec{n} = \sum_{i=1}^{s} (\vec{a}_i - \vec{a}_{i+s}) \} |. 
 \end{equation}
Using double counting, we observe that
\[ \sum_{\vec{n} \in s \mathscr{A} - s \mathscr{A} }  r'_{2s}(\vec{n})^2 = E_{2s,2}(A) = \sum_{\vec{n} \in 2s \mathscr{A}} r_{2s}(\vec{n})^2 . \]
Thus, the functions $r'_{2s}$ and $r_{2s}$ have the same second moment. This is not true when we consider the third moment. In particular, we define
\begin{equation} \label{defpop1}
 E'_{2s,3}(A) = \sum_{\vec{n} \in  s \mathscr{A} - s \mathscr{A} }  r'_{2s}(\vec{n})^3.
 \end{equation}
We see that $E'_{2s,3}(A)$ counts the number of solutions to the system of equations
 \[ \sum_{i=1}^{s} (\vec{a}_i - \vec{a}_{i+s}) =  \sum_{i=1}^{s} (\vec{u}_i - \vec{u}_{i+s}) =  \sum_{i=1}^{s} (\vec{v}_i - \vec{v}_{i+s}) , \]
 with  $\vec{a}_i, \vec{u}_i, \vec{v}_i \in \mathscr{A}$ for each $1 \leq i \leq s$, which is a priori not the same as $E_{2s,3}(A)$. Despite this, we can prove the following variant of Theorem $\ref{newthird}$ for $E'_{2s,3}(A)$. 
 
\begin{lemma} \label{thistoo}
Let $s \geq 2$. Then we have
\[ E'_{2s,3}(A) \ll_{\mathcal{C}}  |A|^{2s-1} E_{2s-1, 2}(A) \log |A| + |A|^{6s-6}.\]
\end{lemma} 
 
\begin{proof}
The proof of this lemma follows in a very similar manner to the proof of Theorem $\ref{newthird}$. Thus, we just briefly sketch the setting of the proof and the main ideas therein. In particular, for each $\vec{u} \in (s-1) \mathscr{A} - s\mathscr{A}$, we consider the curve 
\[ l_{\vec{u}} = \{ (t,\psi(t)) \ | \ t \in \mathbb{R} \} + \vec{u}, \]
with the associated weight function 
\[ w(\vec{u}) = \{ (\vec{a}_1 , \dots, \vec{a}_{2s-1}) \in \mathscr{A}^{2s-1} \ | \ \vec{u}  = \sum_{i=1}^{s-1} \vec{a}_i - \sum_{i=s}^{2s-1} \vec{a}_i \} . \]
We write 
\[ L = \{ l_{\vec{u}} \ | \ \vec{u} \in (s-1) \mathscr{A} - s\mathscr{A} \},\ \text{and} \  P'_{\tau} = \{ \vec{n} \in \mathbb{R}^2 \ | \ \tau \leq  r'_{2s}(\vec{n}) < 2 \tau \}. \]
As before, using Corollary $\ref{wtst}$, we can see that
\begin{equation} \label{rev1rep}
 \tau |P'_{\tau}| \ll_{\mathcal{C}} |P'_{\tau}|^{2/3} |A|^{(2s-1)/3} E_{2s-1,2}(A)^{1/3} + |A|^{2s-1} + |A|^{2s-3}|P'_{\tau}|.
 \end{equation}
This is the analogue of $\eqref{rev1}$ for the quantity $E'_{2s,3}(A)$.
\par

Similarly, we can show that
\begin{align*} \tau |P'_{\tau}|
 & \leq \sum_{\vec{n} \in P'_{\tau}} r'_{2s}(\vec{n}) = \sum_{ \substack{ \vec{a}_1, \dots, \vec{a}_{2s} \in \mathscr{A} \\ \vec{n} \in P'_{\tau} } } \mathds{1}_{ \vec{n} = \vec{a}_1 + \dots - \vec{a}_{2s}} \\
 & = \sum_{\vec{u} \in \mathbb{R}^2} \sum_{ \substack{ \vec{a}_1 \in \mathscr{A} \\ \vec{n} \in P'_{\tau} } } \mathds{1}_{ \vec{n} - \vec{a}_1 = \vec{u}} \sum_{\vec{a_{2}}, \dots, \vec{a_{s}} \in \mathscr{A}}  \mathds{1}_{\vec{u} =  \vec{a}_2 + \dots - \vec{a}_{2s}}.
\end{align*}
We apply Cauchy-Schwarz inequality on the right hand side above to get
\begin{equation} \label{inofrep}
 \tau |P'_{\tau}| \leq E(P'_{\tau}, \mathscr{A})^{1/2} E_{2s-1,2}(A)^{1/2}, 
 \end{equation}
which is the analogue of $\eqref{inof}$ in this case. 
\par

Thus following the proof of Theorem $\ref{kigi}$, and using $\eqref{rev1rep}$ and $\eqref{inofrep}$ in place of $\eqref{rev1}$ and $\eqref{inof}$ respectively, we can prove Lemma $\ref{thistoo}$ mutatis mutandis.
\end{proof}

We note that Lemma $\ref{thistoo}$, together with Theorem $\ref{main}$ implies that
\begin{equation} \label{thistoo1}
E'_{2s,3}(A) \ll_{\mathcal{C}} |A|^{6s-6 + 2^{-2s + 3} } \log |A|,
\end{equation}
whenever $s \geq 2$.


\section{Preliminaries II}

Our aim of this section is to prepare some preliminary results which will aid us in proving higher energy estimates. In this endeavour, we will require upper bounds on the third energy
\[ E_{3,s,X}(A) = \sum_{\vec{n} \in \mathbb{R}^2} r_{s, X}(\vec{n})^3 ,     \]
where
\begin{equation} \label{defr}
 r_{s,X}(\vec{n}) =   | \{ (\vec{a}_1, \dots, \vec{a}_s,\vec{x}) \in \mathscr{A}^s \times X \ | \ \vec{n} = \vec{a}_1 + \dots + \vec{a}_s + \vec{x} \} | 
 \end{equation}
with $s$ being some natural number, and $X$ being some finite, non-empty, large subset of $\mathbb{R}^2$. As before, we see that $E_{3,s,X}(A)$ also counts the numbers of solutions to the system of equations
\[ \vec{x}_1 + \sum_{i=1}^{s} \vec{a}_i = \vec{x}_2 + \sum_{i=1}^{s} \vec{a}_{i+s}  = \vec{x}_3 + \sum_{i=1}^{s} \vec{a}_{i+2s}, \]
with $\vec{x}_1, \vec{x}_2, \vec{x}_3 \in X$ and $\vec{a}_1, \dots, \vec{a}_{3s} \in \mathscr{A}$. Our methods to bound $E_{3,s,X}(A)$ will be similar to the ideas we used in \S5, and thus, we will be considering the set
\[ P_{\tau, s, X} =  \{ \vec{n} \in \mathbb{R}^2 \ | \ \tau \leq r_{s, X}(\vec{n}) < 2 \tau \} \]
where $\tau \geq 1$. With this notation in hand, we can see that
\[ E_{3,s,X}(A) \ll \sum_{j=0}^{J} |P_{2^{j}, s, X}| 2^{3j}, \]
where $J$ is the largest natural number such that
\[ 2^{J} \leq \sup_{\vec{n} \in \mathbb{R}^2} r_{s,X}(\vec{n}). \]
Thus we require upper bounds for $|P_{2^{j}, s, X}|$ for $0 \leq j \leq J$, which in turn, will require upper bounds for quantities of the form
\[ E_{s,X}(A) = \sum_{\vec{n} \in \mathbb{R}^2} r_{s, X}(\vec{n})^2, \ \text{and} \  r_{s,X}(A) = \sup_{\vec{n} \in\mathbb{R}^2} r_{s,X}(\vec{n}). \]
For this purpose, we prove the following lemma.

\begin{lemma} \label{jwt}
Let $|A|^2 \leq |X| \leq  |A|^4$. Then we have
\begin{equation} \label{high1} E_{1,X}(A) \ll_{\mathcal{C}} |A|^{1/2}|X|^{3/2}, \end{equation}
\begin{equation} \label{high2}  E_{2,X}(A) \ll_{\mathcal{C}} |A|^{7/4}|X|^{7/4} \end{equation}
and
\begin{equation} \label{high3} r_{3,X}(A) \ll_{\mathcal{C}} |X|^{2/3} |A|^{4/3}. \end{equation}
\end{lemma}

\begin{proof}
Let $X$ be a finite, non-empty subset of $\mathbb{R}^2$. We first note that
\[ E_{1,X}(A) = \sum_{\vec{n} \in \mathbb{R}^2} r_{1,X}(\vec{n})^2. \]
We define the curve
\[ l_{\vec{x}} = \{(t, \psi(t)) \ | \ t \in \mathbb{R}\} + \vec{x} \]
for each $\vec{x} \in X$, and the set of curves
\[ L = \{ l_{\vec{x}} \ | \ \vec{x} \in X \}. \]
Moreover, we define the point 
\[ p_{\vec{u}} = \vec{u} \]
for each $\vec{u} \in \mathscr{A} + X$, accompanied with the weight function $w(p_{\vec{u}}) = r_{1,X}(\vec{u})$. We finally define our point set
\[ P = \{ p_{\vec{u}} \ | \ \vec{u} \in \mathscr{A} + X \}. \]
We note that
\[ E_{1,X}(A)  \ll_{\mathcal{C}} I_{w}(P, L) \ll (|A||X|)^{1/3} E_{1,X}(A)^{1/3} |X|^{2/3} + |A||X|, \]
using Proposition $\ref{aslo}$, Lemma $\ref{wtst1}$ and the fact that $r_{1,X}(A) \leq |A|$. This implies that
\[ E_{1,X}(A) \ll_{\mathcal{C}} |A|^{1/2}|X|^{3/2} + |A||X|. \]
Combining this with the hypothesis $|X| \geq |A|^2$, we obtain $\eqref{high1}$. 
\par

We prove $\eqref{high2}$ using a similar method. In particular, we define the curve
\[ l_{\vec{v}} = \{ (t, \psi(t)) \ | \ t \in \mathbb{R} \} + \vec{v}, \]
for each $\vec{v} \in \mathscr{A} + X$, accompanied with the weight function $w(l_{\vec{v}}) = r_{1,X}(\vec{v})$. Similarly, we define the point
\[ p_{\vec{u}} = \vec{u}, \]
for each $\vec{u} \in 2\mathscr{A} + X$, accompanied with the weight function $w'(p_{\vec{u}}) = r_{2,X}(\vec{u})$. Writing 
\[ P = \{ p_{\vec{u}} \ | \ \vec{u} \in 2\mathscr{A} + X \} \ \text{and} \ L = \{ l_{\vec{v}} \ | \ \vec{v} \in \mathscr{A} + X \}, \] 
we see that
\begin{align*}
 E_{2,X}(A) \ll I_{w,w'}(P,L) \ll_{\mathcal{C}} &  (|A||X|)^{1/3} E_{1,X}(A)^{1/3} (|A|^2 |X|)^{1/3} E_{2,X}(A)^{1/3} \\
  + & \min\{|A|^2, |X|\} |A||X| + |A|^2|X| \min\{|A|, |X|\},
 \end{align*}
using Proposition $\ref{aslo}$, Lemma $\ref{wtst1}$ and the fact that 
\[ r_{2, X}(A) \leq \min\{ |A|^2, |X|\} \ \text{and} \ r_{1, X}(A) \leq \min\{ |A|, |X|\}. \]
Thus we have
\[ E_{2,X}(A) \ll_{\mathcal{C}} |A|^{3/2}|X| E_{1,X}(A)^{1/2} + \min\{|A|^2, |X|\} |A||X| + |A|^2|X| \min\{|A|, |X|\}.\]
Substituting $\eqref{high1}$ above, we see that
\[ E_{2,X}(A) \ll_{\mathcal{C}}  
\begin{cases}
 |A|^{2} |X|^{2} &\mbox{if } |X| < |A| \\
   |A|^{7/4}|X|^{7/4} + |A|^{3}|X|          &     \mbox{if } |A| \leq |X| < |A|^2  \\
   |A|^{7/4}|X|^{7/4} & \mbox{if } |A|^2 \leq |X|. \end{cases}  \]
Combining this with the hypothesis $|X| \geq |A|^2$, we get $\eqref{high2}$. 
\par

Finally, in order to prove $\eqref{high3}$, it suffices to show that
\[ r_{3,X}(\vec{n}) \ll_{\mathcal{C}} |X|^{2/3} |A|^{4/3} \]
for each $\vec{n} \in \mathbb{R}^2$. Thus, fixing some $\vec{n} \in \mathbb{R}^2$, we consider the point set 
\[ P = \{ \vec{n} - \vec{x} \ | \ \vec{x} \in X \} \]
and the set of curves
\[ L = \{ l_{\vec{v}} \ | \ \vec{v} \in 2\mathscr{A} \} \]
where $l_{\vec{v}}$ is the curve
\[ l_{\vec{v}} = \{( t, \psi(t)) \ | \ t \in \mathbb{R} \} + \vec{v} ,\]
with the weight function $w(\vec{v}) = r_{2}(\vec{v})$ associated with it. Thus, we have
\[ r_{3,X}(\vec{n}) \ll I_w(P,L) \ll_{\mathcal{C}} |X|^{2/3} |A|^{2/3} |A|^{2/3} + |A|^{2} + |X|, \]
using Proposition $\ref{aslo}$, Lemma $\ref{wtst1}$ and the fact that $\sup_{\vec{v} \in \mathbb{R}^2} r_{2}(\vec{v}) \ll_{\mathcal{C}} 1 . $
Thus we have
\[  r_{3,X}(\vec{n})  \ll_{\mathcal{C}} |X|^{2/3} |A|^{4/3} + |A|^{2} + |X|, \]
which when combined with the hypothesis $ |A|^2 \leq |X| \leq |A|^4$ gives us
\[ r_{3,X}(\vec{n}) \ll_{\mathcal{C}} |X|^{2/3} |A|^{4/3},\]
which is the desired bound.
\end{proof}

We will now prove upper bounds for $E_{3,s,X}(A) $ when $s = 2,3$ and when $|X|$ is large. We record these results as the following lemmata. 

\begin{lemma} \label{heb}
Let $|A|^2 \leq |X| \leq |A|^4$. Then we have
\[ E_{3,2,X}(A)  \ll_{\mathcal{C}}  |A|^{3/2} |X|^{5/2} \log |A| + |A|^{5} |X|. \]
\end{lemma}

\begin{proof}

We consider the set of curves
\[ L = \{ l_{\vec{u}} \ | \ \vec{u} \in  \mathscr{A} + X \}, \]
where
\[ l_{\vec{u}}  = \{( t, \psi(t)) \ | \ t \in \mathbb{R} \} + \vec{u} ,\]
with the weight function $w(\vec{u}) = r_{1,X}(\vec{u})$ associated with it. Thus we have
\begin{align*} \tau |P_{\tau, 2, X}|  \ll I_{w}(P_{\tau, 2, X},L)  \ll_{\mathcal{C}}  & \ |P_{\tau, 2, X}|^{2/3}(|A| |X|)^{1/3} E_{1,X}(A)^{1/3} \\  & +  |A| |X| + |P_{\tau, 2, X}|  r_{1, X}(A), 
\end{align*}
by the means of Proposition $\ref{aslo}$ and Lemma $\ref{wtst1}$. Using this along with Lemma $\ref{jwt}$ and the trivial upper bound $r_{1,X}(A) \leq |A|$, we obtain
\[ \tau |P_{\tau, 2, X}|  \ll_{\mathcal{C}}   \ |P_{\tau, 2, X}|^{2/3} |A|^{1/2} |X|^{5/6} +  |A| |X| + |P_{\tau, 2, X}| |A|. \]
Thus there exists a constant $K_1 > 0$ depending on $\mathcal{C}$, such that whenever $\tau \geq K_1 |A|$, we have
\[ |P_{\tau, 2, X}|  \ll_{\mathcal{C}} |A|^{3/2} |X|^{5/2} \tau^{-3} + |A| |X| \tau^{-1}. \]
\par

We define 
\[ U = \{ 0 \leq j \leq J \ | \ 2^{j} < K_1 |A| \} \ \text{and} \ V = \{0, 1, \dots, J\} \setminus U,\]
where $J$ is the largest natural number such that 
\[ 2^{J} \leq \sup_{\vec{n} \in \mathbb{R}^2} r_{2,X}(\vec{n}). \]
We note that
\[ E_{3,2,X}(A) = \sum_{j=0}^{J}  \sum_{\vec{n} \in P_{2^{j}, 2, X} } r_{2, X}(\vec{n})^3  \ll  \sum_{j \in U}  \sum_{\vec{n} \in P_{2^{j}, 2, X} } r_{2, X}(\vec{n})^3 + \sum_{j \in V} |P_{2^{j}, 2, X}| 2^{3j}. \]
We can bound the first term on the right hand side above by
\[   \sum_{j \in U}  \sum_{\vec{n} \in P_{2^{j}, 2, X} } r_{2, X}(\vec{n})^3 \ll_{\mathcal{C}} |A|^{2} \sum_{j \in U}  \sum_{\vec{n} \in P_{2^{j}, 2, X} } r_{2, X}(\vec{n}) \ll_{\mathcal{C}} |A|^{2} |A|^{2} |X| = |A|^{4} |X|. \]
This is stronger than the desired upper bound, and thus, it suffices to consider the case when $j \in V$. In this case, we have
\begin{align*}
\sum_{j \in V} |P_{2^{j}, 2, X}| 2^{3j} 
& \ll_{\mathcal{C}}  \sum_{j \in V} (|A|^{3/2} |X|^{5/2}  + |A| |X|2^{2j}) \\
& \ll_{\mathcal{C}} |A|^{3/2} |X|^{5/2} \log|A| + |A||X| r_{2,X}(A)^{2} \\
&  \ll_{\mathcal{C}} |A|^{3/2} |X|^{5/2} \log |A| + |A|^{5} |X|, 
\end{align*} 
where the last inequality follows from the trivial upper bound $r_{2,X}(A) \leq |A|^{2}$. Thus we see that
\[ E_{3,2,X}(A) \ll_{\mathcal{C}} |A|^{3/2} |X|^{5/2} \log |A| + |A|^{5} |X|, \]
which concludes our proof of Lemma $\ref{heb}$.
\end{proof}

\begin{lemma} \label{imt}
Let $|A|^2 \leq |X| \leq |A|^4$. Then we have
\[ E_{3,3,X}(A)  \ll_{\mathcal{C}} |A|^{15/4} |X|^{11/4} \log|A| + |A|^{14/3} |X|^{7/3}. \]
\end{lemma}

\begin{proof}
While this proof is very similar to the proof of the previous lemma, we provide the necessary details for the sake of completeness. As before, we consider the set of curves
\[ L = \{ l_{\vec{u}} \ | \ \vec{u} \in 2 \mathscr{A} + X \}, \]
where
\[ l_{\vec{u}}  = \{( t, \psi(t)) \ | \ t \in \mathbb{R} \} + \vec{u} ,\]
with the weight function $w(\vec{u}) = r_{2,X}(\vec{u})$ associated with it. Thus we have
\begin{align*} \tau |P_{\tau, 3, X}|  \ll I_{w}(P_{\tau, 3, X},L)  \ll_{\mathcal{C}}  & \ |P_{\tau, 3, X}|^{2/3}(|A|^2 |X|)^{1/3} E_{2,X}(A)^{1/3} \\  & +  |A|^2 |X| + |P_{\tau, 3, X}|  r_{2, X}(A), 
\end{align*}
by the medium of Proposition $\ref{aslo}$ and Lemma $\ref{wtst1}$. Combining this with Lemma $\ref{jwt}$ and the fact that $r_{2, X}(A) \leq |A|^2$, we get
\[  \tau |P_{\tau, 3, X}| \ll_{\mathcal{C}} |P_{\tau, 3, X}|^{2/3} |A|^{5/4} |X|^{11/12} +  |A|^2 |X| + |P_{\tau, 3, X}| |A|^2 .\]
 Thus there exists a constant $K_1 > 0$ depending on $\mathcal{C}$, such that whenever $\tau \geq K_1 |A|^2$, we have
  \[ \tau |P_{\tau, 3, X}| \ll_{\mathcal{C}} |P_{\tau, 3, X}|^{2/3} |A|^{5/4} |X|^{11/12} +  |A|^2 |X| ,\]
which consequently implies that
\begin{equation} \label{iofire}
 |P_{\tau, 3, X}| \ll_{\mathcal{C}} |A|^{15/4} |X|^{11/4} \tau^{-3} + |A|^{2} |X| \tau^{-1}. 
 \end{equation}
\par

We write 
\[ U = \{ 0 \leq j \leq J \ | \ 2^{j} < K_1 |A|^2 \} \ \text{and} \ V = \{0, 1, \dots, J\} \setminus U. \]
Thus, we have
\[ E_{3,3,X}(A) = \sum_{j=0}^{J}  \sum_{\vec{n} \in P_{2^{j}, 3, X} } r_{3, X}(\vec{n})^3  \ll  \sum_{j \in U}  \sum_{\vec{n} \in P_{2^{j}, 3, X} } r_{3, X}(\vec{n})^3 + \sum_{j \in V} |P_{2^{j}, 3, X}| 2^{3j}. \]
We note that
\[ \sum_{j \in U}  \sum_{\vec{n} \in P_{2^{j}, 3, X} } r_{3, X}(\vec{n})^3 \ll_{\mathcal{C}} |A|^{4} |A|^{3} |X| = |A|^{7}|X|. \]
Moreover,
\begin{align*}
 \sum_{j \in V} \sum_{\vec{n} \in P_{2^{j}, 3, X} } r_{3, X}(\vec{n})^3 
 & \ll_{\mathcal{C}} \sum_{j \in V} |P_{2^{j}, 3, X}| 2^{3j} \ll_{\mathcal{C}} \sum_{j=0}^{J} \big( |A|^{15/4} |X|^{11/4}    + |A|^{2} |X| 2^{2J} \big) \\
 & \ll_{\mathcal{C}} |A|^{15/4} |X|^{11/4} \log|A| + |A|^2 |X|  r_{3,X}(A)^2 \\
 & \ll_{\mathcal{C}} |A|^{15/4} |X|^{11/4} \log|A| + |A|^{14/3} |X|^{7/3} , 
 \end{align*}
 with the last inequality following from $\eqref{high3}$. Thus, we have
 \[  E_{3,3,X}(A)\ll_{\mathcal{C}} |A|^{15/4} |X|^{11/4} \log|A| + |A|^{14/3} |X|^{7/3} + |A|^{7}|X|. \]
 As 
 \[ |A|^{7}|X| \leq |A|^{14/3} |X|^{7/3} \]
 whenever $|X| \geq |A|^{7/4}$, using the hypothesis $|A|^2 \leq |X| \leq |A|^4$, we see that
 \[  E_{3,3,X}(A)\ll_{\mathcal{C}} |A|^{15/4} |X|^{11/4} \log|A| + |A|^{14/3} |X|^{7/3},\]
which is the desired bound.
\end{proof}

In the next section, we will prove stronger lower bounds for the set $s \mathscr{A} - s\mathscr{A}$ when $s = 2,3$, using higher energy methods. A key ingredient in these proofs will be Lemma $\ref{heb}$ and Lemma $\ref{imt}$ respectively.


\section{Higher energy estimates for $|s \mathscr{A} - s\mathscr{A}|$}

Let $s \geq 2$ be a natural number. In this section, we will consider lower bounds for the size of the set $s \mathscr{A} - s\mathscr{A}$, and so, we will prove Theorem $\ref{highsum}$. To do this, we will generalise an approach used in \cite{RS2019} to $s$-fold energies and sumsets. For ease of notation, for each $\vec{m} \in \mathbb{R}^4$, we will write $\vec{m} = (\vec{m}_1 , \vec{m}_2)$ where $\vec{m}_1, \vec{m}_2 \in \mathbb{R}^2$. 
\par

At the core of the proof of Theorem $\ref{highsum}$ lies the following lemma. 

\begin{lemma} \label{punchline}
Let $T_{2s}(A)$ count the number of solutions to the following system of equations
\begin{equation} \label{deftoo}
  \sum_{i=1}^{s} (\vec{a}_i - \vec{u}_i) =  \sum_{i=s+1}^{2s} (\vec{a}_i - \vec{u}_i) \ \text{and} \ 
  \sum_{i=1}^{s} (\vec{a}_i - \vec{v}_i) =  \sum_{i=s+1}^{2s} (\vec{a}_i - \vec{v}_i),
\end{equation}
 with $\vec{a}_i, \vec{u}_i, \vec{v}_i \in \mathscr{A}$, for each $1 \leq i \leq 2s$. Then we have
\[  E_{3,s,s \mathscr{A} - s\mathscr{A}}(A)^{1/2}  |s\mathscr{A} - s\mathscr{A}|^{3/2} \gg  |A|^{15s/2}T_{2s}(A)^{-1} (\log |A| )^{-3}. \]
\end{lemma}
 
Firstly, we note that we can rewrite $T_{2s}(A)$ as 
\[ T_{2s}(A) = \sum_{\vec{m} \in \mathbb{R}^4} R_{2s}(\vec{m})^2 \]
where $R_{2s}(\vec{m})$ counts the number of solutions to the following system of equations
\begin{align*}
\vec{m}_1  = \sum_{i=1}^{s} (\vec{a}_i - \vec{u}_i) \ \text{and} \ \vec{m}_2  = \sum_{i=1}^{s} (\vec{a}_i- \vec{v}_i)
\end{align*}
with $\vec{a}_i, \vec{u}_i, \vec{v}_i \in \mathscr{A}$ for each $1 \leq i \leq s$. Moreover, upon rearrangement, we see that $\eqref{deftoo}$ is equivalent to the system of equations
 \[ \sum_{i=1}^{s} (\vec{a}_i - \vec{a}_{i+s}) =  \sum_{i=1}^{s} (\vec{u}_i - \vec{u}_{i+s}) =  \sum_{i=1}^{s} (\vec{v}_i - \vec{v}_{i+s}) , \]
 with  $\vec{a}_i, \vec{u}_i, \vec{v}_i \in \mathscr{A}$ for each $1 \leq i \leq 2s$. This is a third energy type estimate, and in particular, 
 \[ T_{2s}(A) = \sum_{\vec{n} \in \mathbb{R}^2} r'_{2s}(\vec{n})^{3} = E'_{2s,3}(A), \]
 where $r'_{2s}(\vec{n})$ and $E'_{2s,3}(A)$ are defined as in $\eqref{defpop}$ and $\eqref{defpop1}$ respectively. 
\par

Next, we observe that
\[ \sum_{\vec{n} \in s\mathscr{A} - s\mathscr{A}} r'_{2s}(\vec{n}) = |A|^{2s}. \]
Using a dyadic pigeonhole principle, we can deduce that exists a real number $\Delta > 0$ and a subset $P \subseteq s \mathscr{A} - s \mathscr{A}$ such that for each $\vec{n} \in P$, we have
\begin{equation} \label{fist1}
 \Delta \leq r'_{2s}(\vec{n}) < 2 \Delta, 
 \end{equation}
and
\[ \Delta |P| \gg |A|^{2s} (\log |A|)^{-1}. \]
This further implies that
\begin{equation} \label{fist2}
 \Delta \gg |A|^{2s} |P|^{-1} (\log |A|)^{-1} \gg |A|^{2s} |s \mathscr{A} - s\mathscr{A}|^{-1} (\log |A|)^{-1}. 
 \end{equation} 
We define the hypergraph $G \subseteq \mathscr{A}^{2s}$ as 
\[ G = \{ (\vec{a}_1, \dots, \vec{a}_{2s}) \in \mathscr{A}^{2s} \ | \ \sum_{i=1}^{s} (\vec{a_{i}} - \vec{a}_{i+s}) \in P \}. \]
Thus we have
\[ |G| \geq |P| \Delta \gg |A|^{2s} (\log |A|)^{-1}. \]
For each $\vec{m} \in \mathbb{R}^4$, we use $R'_{2s}(\vec{m})$ to denote the number of solutions to
\begin{align*}
\vec{m}_1  = \sum_{i=1}^{s} (\vec{a}_i - \vec{u}_i) \ \text{and} \ \vec{m}_2  = \sum_{i=1}^{s} (\vec{a}_i- \vec{v}_i)
\end{align*}
with $\vec{a}_i \in \mathscr{A}$ for each $1 \leq i \leq s$, and $(\vec{u}_1, \dots, \vec{u}_{s}, \vec{v}_1, \dots, \vec{v}_{s}) \in G$. Finally, we use $U$ to denote the set
	\[ U = \{ \vec{m} \in \mathbb{R}^4 \ | \ R'_{2s}(\vec{m}) \geq 1 \}.  \]
We have finished our set up for the proof of Lemma $\ref{punchline}$. 
\par

\begin{proof}[Proof of Lemma $\ref{punchline}$]
We first note that
\[ \sum_{\vec{m} \in U} R'_{2s}(\vec{m}) = |A|^{s} |G| \gg  |A|^{3s} (\log |A|)^{-1}. \]
Applying Cauchy-Schwarz inequality, we obtain
\begin{equation} \label{sar}
 |U| \sum_{\vec{m} \in U} R'_{2s}(\vec{m})^2 \gg |A|^{6s} (\log |A|)^{-2}. 
 \end{equation}
It is evident that for each $\vec{m} \in \mathbb{R}^4$, we have
\[ R'_{2s}(\vec{m}) \leq R_{2s}(\vec{m}), \]
whence,
\[ \sum_{\vec{m} \in U} R'_{2s}(\vec{m})^2  \leq \sum_{\vec{m} \in \mathbb{R}^4} R_{2s}(\vec{m})^2 = T_{2s}(A). \]
Substituting this into $\eqref{sar}$, we find that
\begin{equation} \label{ah}
|U| \gg |A|^{6s} (\log |A| )^{-2} T_{2s}(A)^{-1}. 
\end{equation}
\par

Given $\vec{m} \in U$, we can write
\[ \vec{m}_1  = \sum_{i=1}^{s} (\vec{a}_i - \vec{u}_i) \ \text{and} \ \vec{m}_2  = \sum_{i=1}^{s} (\vec{a}_i- \vec{v}_i), \]
for some $\vec{a}_i \in \mathscr{A}$ for each $1 \leq i \leq s$, and for some $(\vec{u_1}, \dots, \vec{u_{s}}, \vec{v_1}, \dots, \vec{v_{s}}) \in G$. Thus, 
\[ \vec{m}_2 - \vec{m}_1 = \sum_{i=1}^{s}  (\vec{u}_i - \vec{v}_i). \]
Combining this with the definition of $G$, we see that $\vec{m}_2 - \vec{m}_1 \in P$. Recalling $\eqref{fist1}$, we see that
\[ r'_{2s}(\vec{m}_2 - \vec{m}_1) \geq \Delta.  \]
Summing this over all elements in $U$, we get
\[ \sum_{\vec{m} \in U} r'_{2s}(\vec{m}_2 - \vec{m}_1) \geq \Delta |U|, \]
which together with $\eqref{fist2}$ and $\eqref{ah}$ gives us
\begin{equation} \label{phew}
 \sum_{\vec{m} \in U} r'_{2s}(\vec{m}_2 - \vec{m}_1) \gg |A|^{8s} T_{2s}(A)^{-1}|s \mathscr{A} - s\mathscr{A}|^{-1} (\log |A| )^{-3}. 
\end{equation}
\par

We note that for any choice of $\vec{m} \in U$, the elements $\vec{m}_1, \vec{m}_2$ lie in $s\mathscr{A} - s\mathscr{A}$. Moreover, we see that any choice of $\vec{m}_1$ and $\vec{m}_2$ fixes $\vec{m}$, and vice versa. Thus 
\[ \sum_{\vec{m} \in U} r'_{2s}(\vec{m}_2 - \vec{m}_1) \]
is bounded above by the number of solutions to the system of equations
\[ \vec{m}_2 - \vec{m}_1 = \sum_{i=1}^{s} (\vec{a}_i - \vec{a}_{i+s}), \] 
with $\vec{m}_1, \vec{m}_2 \in s\mathscr{A} - s\mathscr{A}$ and $\vec{a}_1, \dots, \vec{a}_{2s} \in \mathscr{A}$. This can be written as
\[ \sum_{\vec{n} \in \mathbb{R}^2 } r_{s, s\mathscr{A} - s\mathscr{A}}(\vec{n})^2, \]
where $r_{s, s\mathscr{A} - s\mathscr{A}}(\vec{n})$ is defined as in $\eqref{defr}$. Using Cauchy-Schwarz inequality, we find that
\begin{align*}
 \sum_{\vec{n} \in \mathbb{R}^2 } r_{s, s\mathscr{A} - s\mathscr{A}}(\vec{n})^2 
 & \leq  \big( \sum_{\vec{n} \in \mathbb{R}^2 } r_{s, s\mathscr{A} - s\mathscr{A}}(\vec{n})^3 \big)^{1/2} \big( \sum_{\vec{n} \in \mathbb{R}^2 } r_{s, s\mathscr{A} - s\mathscr{A}}(\vec{n}) \big)^{1/2} \\
 & =  E_{3,s,s \mathscr{A} - s\mathscr{A}}(A)^{1/2} |A|^{s/2} |s\mathscr{A} - s\mathscr{A}|^{1/2}. 
\end{align*}
Combining this with the preceding inequalities, we see that
\begin{equation} \label{waidt}
 \sum_{\vec{m} \in U} r'_{2s}(\vec{m}_2 - \vec{m}_1) \leq E_{3,s,s \mathscr{A} - s\mathscr{A}}(A)^{1/2} |A|^{s/2} |s\mathscr{A} - s\mathscr{A}|^{1/2}.
 \end{equation}
Substituting $\eqref{waidt}$ into $\eqref{phew}$, we get
\begin{equation*} 
  E_{3,s,s \mathscr{A} - s\mathscr{A}}(A)^{1/2}  |s\mathscr{A} - s\mathscr{A}|^{3/2} \gg  |A|^{15s/2}T_{2s}(A)^{-1} (\log |A| )^{-3},
  \end{equation*}
  which finishes the proof of Lemma $\ref{punchline}$. 
\end{proof}

One final ingredient that we require before proceeding with the proof of Theorem $\ref{highsum}$ is an estimate for $T_{2s}(A)$. As before, we note that $T_{2s}(A) = E'_{2s,3}(A)$ and thus, we can use $\eqref{thistoo1}$ to deduce that
\begin{equation} \label{thistoo2}
 T_{2s}(A) \ll_{\mathcal{C}} |A|^{6s-6 + 2^{-2s + 3} } \log |A|, 
 \end{equation}
whenever $s \geq 2$. We now prove Theorem $\ref{highsum}$. 

\begin{proof}[Proof of Theorem $\ref{highsum}$]
Let $s= 2$. We note that Theorem $\ref{main}$ implies that
\begin{equation} \label{thresh2}
 |2 \mathscr{A} - 2 \mathscr{A} | \gg_{\mathcal{C}} |A|^{3-1/4} \geq |A|^{2}, 
 \end{equation}
when $|A|$ is large enough. Moreover, we trivially have
\[ |2 \mathscr{A} - 2 \mathscr{A} |  \leq |A|^{4}. \]
Thus we can use Lemma $\ref{heb}$ to see that
\[ E_{3,2, 2\mathscr{A} - 2 \mathscr{A}} \ll_{\mathcal{C}} |A|^{3/2} |2 \mathscr{A} - 2 \mathscr{A}|^{5/2} \log |A| + |A|^{5} | 2 \mathscr{A} - 2 \mathscr{A}|. \]
Using $\eqref{thresh2}$, we can deduce that
\[ |2 \mathscr{A} - 2\mathscr{A}|  \geq |A|^{2 + 1/3}, \]
whenever $A$ is large, which in turn implies that
\[ |A|^{3/2} |2 \mathscr{A} - 2 \mathscr{A}|^{5/2} \geq  |A|^{5} | 2 \mathscr{A} - 2 \mathscr{A}|.  \]
Thus, we have
\[ E_{3,2, 2\mathscr{A} - 2 \mathscr{A}} \ll_{\mathcal{C}} |A|^{3/2} |2 \mathscr{A} - 2 \mathscr{A}|^{5/2} \log |A|. \]
We combine this with Lemma $\ref{punchline}$ and estimate $\eqref{thistoo2}$ to obtain
\[ |A|^{3/4}|2 \mathscr{A} - 2 \mathscr{A}|^{5/4}|2 \mathscr{A} - 2 \mathscr{A}|^{3/2} (\log |A|)^{1/2} \gg_{\mathcal{C}} |A|^{9- 1/2}  (\log |A|)^{-4}. \]
This implies that
\[ |2 \mathscr{A} - 2 \mathscr{A}|^{11/4} \gg_{\mathcal{C}} |A|^{31/4} (\log |A|)^{-9/2}, \]
and hence, we have
\[ |2 \mathscr{A} - 2 \mathscr{A}| \gg_{\mathcal{C}} |A|^{31/11}(\log |A|)^{-18/11} = |A|^{3 - 2/11}(\log |A|)^{-18/11}  , \]
which is the required bound. 
\par

Similarly, let $s=3$. We can use Theorem $\ref{main}$ to see that
\begin{equation} \label{thresh3}
| 3 \mathscr{A} - 3 \mathscr{A}| \gg_{\mathcal{C}} |A|^{3- 1/16}. 
\end{equation}
Moreover, we can assume that
\[ | 3 \mathscr{A} - 3 \mathscr{A}| \leq |A|^{4}, \]
since otherwise, we will have a much stronger bound than we require. We can now use Lemma $\ref{imt}$ to infer that
\[ E_{3,3,3\mathscr{A} - 3 \mathscr{A}}(A)  \ll_{\mathcal{C}} |A|^{15/4} |3\mathscr{A} - 3 \mathscr{A}|^{11/4} \log|A| + |A|^{14/3} |3\mathscr{A} - 3 \mathscr{A}|^{7/3}. \]
It is evident from $\eqref{thresh3}$ that we have
\[ |3 \mathscr{A} - 3 \mathscr{A}| \geq |A|^{2 + 1/5} \]
whenever $A$ is large enough, which in turn gives us
\[ |A|^{15/4} |3\mathscr{A} - 3 \mathscr{A}|^{11/4} \log|A| \geq |A|^{14/3} |3\mathscr{A} - 3 \mathscr{A}|^{7/3}.  \]
Thus, we see that
\[ E_{3,3,3\mathscr{A} - 3 \mathscr{A}}(A)  \ll_{\mathcal{C}} |A|^{15/4} |3\mathscr{A} - 3 \mathscr{A}|^{11/4} \log|A|.\]
As in the case $s=2$, we use this along with Lemma $\ref{punchline}$ and $\eqref{thistoo2}$ to find that
\[ |A|^{15/8} |3\mathscr{A} - 3 \mathscr{A}|^{11/8} |3 \mathscr{A} - 3 \mathscr{A}|^{3/2} (\log|A|)^{1/2} \gg_{\mathcal{C}} |A|^{9/2 + 6 - 1/8} (\log |A| )^{-4}, \]
and consequently, we get
\[ |3\mathscr{A} - 3 \mathscr{A}|^{23/8} \gg_{\mathcal{C}} |A|^{68/8} (\log|A|)^{-9/2}. \]
This delivers the bound
\[ |3\mathscr{A} - 3 \mathscr{A}| \gg_{\mathcal{C}} |A|^{68/23} (\log|A|)^{-36/23} = |A|^{3 - 1/23} (\log|A|)^{-36/23} , \]
and thus we conclude the proof of Theorem $\ref{highsum}$. 
\end{proof}


\section{Higher energy estimates for $E_{4,2}(A)$}

We can now refine the estimates for $E_{s,2}(A)$ as provided by Theorem $\ref{main}$, whenever $s \geq 4$. In view of Theorem $\ref{thiswillpass}$, it is natural to start with the case $s=4$. Thus, we record the following theorem, proving which will be our main goal of this section.

\begin{theorem} \label{hero}
Let $c = 1/7246$. Then for all finite, non-empty sets $A \subseteq I$ that are large enough in terms of $\mathcal{C}$, we have
\[ E_{4,2}(A) \leq |A|^{5 + 1/4 - c}. \]
\end{theorem}

\begin{proof}
Note that it suffices to prove that whenever 
\begin{equation} \label{clearsky}
E_{4,2}(A) > |A|^{5 + 1/4 - c}, 
 \end{equation}
we have that $|A| \ll_{\mathcal{C}} 1$. We begin this endeavour by writing $B = 2 \mathscr{A}$, whenceforth, we note that
\begin{equation} \label{clearskies}
 |A|^{2} \ll_{\mathcal{C}} |B| \leq |A|^{2}.
 \end{equation}
For ease of notation, we will write $X \gs Y$, or equivalently $Y \ls X$, to mean $|X| \geq C|Y| (\log{|A|})^{D}$ where $C$ and $D$ are constants, and $C>0$. The corresponding notation $X \gs_{z} Y$ is defined similarly, except in this case the constants $C$ and $D$ will depend on the parameter $z$. Using $\eqref{clearskies}$, we infer that 
\[ \log|A| \ll_{\mathcal{C}} \log|B| \ll \log |A|, \]
and thus, in the definition of the $\gs_{\mathcal{C}}$ notation, we can replace powers of $\log |A|$ by powers of $\log |B|$. 
\par

We define $E_{2}(B)$ and $E_{3}(B)$ as 
\[ E_{2}(B) = \{ (b_1, \dots, b_4) \in B^4 \ | \ b_1 + b_2 = b_3 + b_4 \}, \]
and
\[ E_{3}(B) = \{  (b_1, \dots, b_6) \in B^6 \ | \ b_1 + b_2 = b_3 + b_4 = b_5 + b_6 \}. \]
It is trivial to see that
\[ E_{2}(B) \leq E_{4,2}(A) \ \text{and} \ E_{3}(B) \leq E_{4,3}(A). \]
Moreover, we recall that for fixed $(n_1, n_2) \in \mathbb{R}^2$, there are $O_{\mathcal{C}}(1)$ choices of $a_1, a_2 \in A$ such that
\[ a_1 + a_2 = n_1 \ \text{and} \ \psi(a_1) + \psi(a_2) = n_2. \]
Thus we have
\[ E_{4,2}(A) \ll_{\mathcal{C}} E_{2}(B) \ \text{and} \ E_{4,3}(A) \ll_{\mathcal{C}} E_{3}(B). \]
\par

The preceding inequality along with $\eqref{clearsky}$ and $\eqref{clearskies}$ implies that 
\begin{equation} \label{csk1}
 E_{2}(B) \gg_{\mathcal{C}} |A|^{5 + 1/4 - c} \gg_{\mathcal{C}} |B|^{5/2 + 1/8 - c/2}. 
 \end{equation}
Furthermore, Theorem $\ref{kigi}$ furnishes the following bound
\begin{equation} \label{csk2}
  E_{3}(B) \leq E_{4,3}(A) \ll_{\mathcal{C}} |A|^{6 + 1/2} \log |A| \ll_{\mathcal{C}} |B|^{ 3 + 1/4} \log |B|. 
  \end{equation}
We now use a result proved by Shkredov \cite[Theorem $1.3$]{Sh2013} which relates the second energy, the third energy and the sumset. 

\begin{lemma} \label{shk}
Let $G$ be an abelian group, and let $B \subseteq G$ be a finite, non-empty set. Let $K$ and $M$ satisfy
\[ E_{2}(B) = |B|^3/K \ \text{and} \ E_{3}(B) = M|B|^{4}/K^{2} . \]
Then there exists a set $B' \subseteq B$ such that
\begin{equation} \label{tide1}
 |B'| \gg M^{-10} (\log M)^{-15} |B|, 
 \end{equation}
holds, and for every $n, m \in \mathbb{N}$, we have
\begin{equation} \label{tide2}
 |nB' - mB'| \ll M^{54(n+m)} (\log M)^{84(n+m)} K |B'|. 
 \end{equation}
\end{lemma}

Defining $K$ and $M$ as in Lemma $\ref{shk}$, we use $\eqref{csk1}$ to see that
\[|B|^{5/2 + 1/8 - c/2} \ll_{\mathcal{C}}  |B|^{3} / K , \]
whence, 
\[ K \ll_{\mathcal{C}} |B|^{1/2 - 1/8 + c/2}. \]
Moreover, $\eqref{csk2}$ implies that
\[ M|B|^{4}/K^2 \ll_{\mathcal{C}} |B|^{ 3 + 1/4} \log |B|, \]
and consequently, we have
\[ M \ll_{\mathcal{C}} |B|^{-3/4} K^{2} \log |B| \ll_{\mathcal{C}} |B|^{c} \log |B|. \]
We now use Lemma $\ref{shk}$ to deduce that there exists a set $B' \subseteq B$ such that $B'$ satisfies $\eqref{tide1}$, and $\eqref{tide2}$ with $n,m = 2$. Thus, we have
\begin{equation} \label{tada1}
 |B'| \gs_{\mathcal{C}} |B|^{1 - 10 c} , 
 \end{equation}
and
\begin{equation} \label{tada2}
 |2B' - 2B'| \ls_{\mathcal{C}} |B|^{216 c}|B|^{1/2 - 1/8 + c/2} |B'|.
 \end{equation}
\par

We observe that
\[ B' \subseteq 2 \mathscr{A} \subseteq \cup_{\vec{a} \in \mathscr{A}} (\vec{a} + \mathscr{A}), \]
and thus, there exists some $\vec{a} \in \mathscr{A}$ such that
\[ | B' \cap (\vec{a} + \mathscr{A})| \geq |B'| |A|^{-1} . \] 
Combining this with $\eqref{tada1}$ and $\eqref{clearskies}$, we see that
\[  | B' \cap (\vec{a} + \mathscr{A})|  \gs_{\mathcal{C}} |B|^{1 - 10 c} |A|^{-1} \gs_{\mathcal{C}} |A|^{1 - 20 c}. \]
We use $A'$ to denote the subset of $A$ that satisfies
\[ \mathscr{A'} = \mathscr{A} \cap (B' - \vec{a}), \]
where $\mathscr{A'} = \{ (a, \psi(a) ) \ | \ a \in A' \}$. The preceding inequality then implies that
\begin{equation} \label{banana}
 |A' | \gs_{\mathcal{C}} |A|^{1- 20 c}. 
 \end{equation}
It is evident from $\eqref{tada2}$ that we have
\begin{align*}
 |2(B' - \vec{a})- 2 (B' - \vec{a})| = |2B' - 2B'| 
 & \ls_{\mathcal{C}} |B|^{216 c}|B|^{1/2 - 1/8 + c/2} |B'| \\
 & \ls_{\mathcal{C}} |A|^{432 c} |A|^{1- 1/4 + c} |A|^2 .
 \end{align*}
Consequently, we can deduce that
\[ |2 \mathscr{A'} - 2\mathscr{A'}| \leq |2(B' - \vec{a})- 2 (B' - \vec{a})|  \ls_{\mathcal{C}} |A|^{432 c} |A|^{3 - 1/4 + c} . \]
Finally, Theorem $\ref{highsum}$, together with $\eqref{banana}$, implies that
\[ |2 \mathscr{A'} - 2\mathscr{A'}| \gs_{\mathcal{C}} |A'|^{3 - 2/11} \gs_{\mathcal{C}} |A|^{3-2/11 -60c}. \] 
Thus we have
\[  |A|^{3-2/11 -  60c} \ls_{\mathcal{C}}  |A|^{433 c} |A|^{3 - 1/4} , \]
from which we infer that
\[ |A|^{1/4 - 2/11} \ls_{\mathcal{C}} |A|^{493 c}. \]
Using the fact that $(\log |A|)^{O_{\mathcal{C}}(1)} \ll_{\mathcal{C}} |A|^{c}$, we get that
\[ |A|^{3/44} \ll_{\mathcal{C}} |A|^{494 c}. \] 
Since $c = 1/7246$, we have that $3/44 > 494c$, whence the conclusion $|A| \ll_{\mathcal{C}} 1$ follows. This finishes our proof of Theorem $\ref{hero}$.
\end{proof}

In the next section, we will use Theorem $\ref{highsum}$ and Theorem $\ref{hero}$ to improve various estimates on the sumset, the second energy and the third energy.


\section{Higher energy estimates for larger values of $s$}

Our first goal of this section is to bootstrap Theorem $\ref{hero}$ and use it to break the threshold bound for $E_{s,2}(A)$ for all values of $s \geq 5$. To see this, we recall Theorem $\ref{thiswillpass}$ which states that for all $s \geq 3$, we have
\[ E_{s,2}(A) \ll_{\mathcal{C}} |A|^{s-1/2}E_{s-1,2}(A)^{1/2} + |A|^{2s-3}. \]
In \S4, we used this estimate to get upper bounds for $E_{s,2}(A)$ using the fact that $E_{2,2}(A) \ll_{\mathcal{C}} |A|^2$ as our base case. We can now strengthen upper bounds for $E_{s,2}(A)$ when $s \geq 5$, using Theorem $\ref{hero}$ as our base case instead.

\begin{theorem} \label{genen}
Let $s \geq 4$ and $c = 1/7246$. Then we have
\[ E_{s,2}(A) \ll_{\mathcal{C}} |A|^{2s - 3 + (1/4 - c) \cdot 2^{-s+4}}. \]
\end{theorem}

\begin{proof}
We prove our result inductively. Our base case is when $s=4$, which is implied by Theorem $\ref{hero}$. Thus it suffices to prove the inductive step, and so we assume that $s > 4$. Then by Theorem $\ref{thiswillpass}$, we have
\[ E_{s,2}(A) \ll_{\mathcal{C}} |A|^{s-1/2}E_{s-1,2}(A)^{1/2} + |A|^{2s-3}. \]
As $|A|^{2s - 3}$ is asymptotically smaller than our required upper bound, it is sufficient to consider the case when 
\[ E_{s,2}(A) \ll_{\mathcal{C}} |A|^{s-1/2} E_{s-1,2}(A)^{1/2}. \]
By the inductive hypothesis, we have
\[ E_{s-1,2}(A) \ll_{\mathcal{C}} |A|^{2s - 5 + (1/4 - c) \cdot 2^{-s+5}}, \]
which when substituted into the preceding inequality, gives us
\[ E_{s,2}(A)  \ll_{\mathcal{C}} |A|^{s - 1/2 + s - 5/2 +  (1/4 - c) \cdot 2^{-s+4}} = |A|^{2s - 3 +(1/4 - c) \cdot 2^{-s+4}}. \]
Thus, we have proven Theorem $\ref{genen}$ by induction. 
\end{proof}

It is clear that when $s \geq 4$, Theorem $\ref{genen}$ strengthens our upper bounds for $E_{s,2}(A)$, and improves upon Theorem $\ref{main}$ by a factor of $|A|^{c \cdot 2^{-s+4}}$. We can prove a similar result for lower bounds for $|s \mathscr{A} - s\mathscr{A}|$. We first prove a preliminary lemma, which we will iterate to obtain our lower bound.

\begin{lemma} \label{idy}
Let $S \subseteq \mathbb{R}^2$ be a finite, non-empty set. Then we either have
\[ |\mathscr{A} + S| \gg_{\mathcal{C}} |A|^{3/2} |S|^{1/2} \ \text{or} \ |\mathscr{A} + S| \gg_{\mathcal{C}} |S||A|. \]
\end{lemma}

\begin{proof}
As before, we begin by defining our set $P$ of points as 
\[ P = \mathscr{A} + S, \]
and our set $L$ of curves
\[ L = \{ l_{\vec{s}} \ | \ \vec{s} \in S\}, \]
where 
\[ l_{\vec{s}} = \{ (t,\psi(t)) \ | \ t \in \mathbb{R} \} + \vec{s} . \]
We can assume that $|S| \geq 2$, since otherwise the trivial bound $|\mathscr{A} + S| \geq |A|$ suffices. Hence, we can now use Proposition $\ref{aslo}$ to see that $L$ is a $\mathcal{C}$-valid collection of curves. Since both our set of points and set of curves are unweighted, we can set $w(l) = 1$ for each $l \in L$ and $w'(p) = 1$ for each $p \in P$. Thus, we apply Lemma $\ref{wtst1}$ to obtain 
\[ I(P,L) \ll_{\mathcal{C}} |P|^{2/3}|L|^{2/3} + |P| + |L|. \]
Moreover, we see that for a fixed $\vec{s} \in S$, the curve $l_{\vec{s}}$ contains the set of points $\mathscr{A} + \vec{s}$ which is a subset of $P$, thus implying that
\[ I(P,L) \geq \sum_{\vec{s} \in S} |\mathscr{A} + \vec{s}| = |A| |S| = |A| |L|. \]
Combining the two inequalities, we find that
\[ |L||A| \ll_{\mathcal{C}} |P|^{2/3}|L|^{2/3} + |P| + |L|. \]
As $|A|$ can be sufficiently large in terms of $\mathcal{C}$, we get
\[|L||A| \ll_{\mathcal{C}} |P|^{2/3}|L|^{2/3} + |P|. \]
In particular, this implies that either
\[ |\mathscr{A} + S| = |P| \gg_{\mathcal{C}} |L||A| = |S||A|, \]
or
\[ |\mathscr{A} + S| = |P| \gg_{\mathcal{C}} |A|^{3/2}|L|^{1/2} = |A|^{3/2}|S|^{1/2}. \] 
This finishes the proof of Lemma $\ref{idy}$.
\end{proof}

We now present improved bounds for $s \mathscr{A} - s\mathscr{A}$.

\begin{theorem} \label{gensum}
Let $s \geq 3$. Then we have
\[ |s \mathscr{A} - s\mathscr{A}| \gg_{\mathcal{C}} |A|^{ 3 - \delta \cdot 4^{-s+3}} (\log|A|)^{ -C \cdot 4^{-s+3}},\]
where $\delta = 1/23$ and $C = 36/23$. 
\end{theorem}
\begin{proof}
As in the case of Theorem $\ref{genen}$, we proceed inductively. When $s=3$, we see that Theorem $\ref{gensum}$ is implied by Theorem $\ref{highsum}$, and thus, it suffices to prove the inductive step. Let $s \geq 4$ and set $S = (s-1)\mathscr{A} - (s-1)\mathscr{A}$. By the inductive hypothesis, we have
\[ |S| \gg_{\mathcal{C}} |A|^{ 3 -  \delta \cdot 4^{-s+4}} (\log|A|)^{-C \cdot 4^{-s+4}}. \]
Since $|S| > |A|$, we have $|S|^{1/2} |A|^{3/2} < |S| |A|$, whence, Lemma $\ref{idy}$ gives us
\[ |S + \mathscr{A}| \gg_{\mathcal{C}} |A|^{3/2} |S|^{1/2}  \gg_{\mathcal{C}}  |A|^{3  -  \delta \cdot 4^{-s+4}\cdot 2^{-1} } (\log|A|)^{-C \cdot 4^{-s+4} \cdot 2^{-1} }. \]
We apply Lemma $\ref{idy}$ again to discern that
\[ |S + \mathscr{A} - \mathscr{A}| =  |\mathscr{A} - (S + \mathscr{A})| \gg_{\mathcal{C}} |A|^{3/2} |S + \mathscr{A}|^{1/2} 
\gg_{\mathcal{C}} |A|^{3 -  \delta \cdot 4^{-s+3} } (\log|A|)^{-C \cdot 4^{-s+3}}. \]
This finishes the inductive step, and so, Theorem $\ref{gensum}$ holds true.
\end{proof}

We remark that if we instead used $s=2$ as our base case, we would have shown that
\[ |s \mathscr{A} - s\mathscr{A}| \gs_{\mathcal{C}} |A|^{ 3 - \delta \cdot 4^{-s+2}}, \]
where $\delta = 2/11$. In particular, when $s=3$, this would have implied that
\[ |3 \mathscr{A} - 3\mathscr{A}| \gs_{\mathcal{C}} |A|^{3 - 1/22}, \] 
which is weaker than Theorem $\ref{highsum}$ where we directly applied higher energy methods to estimate lower bounds for $|3 \mathscr{A} - 3\mathscr{A}|$. In general for larger $s$, dealing directly with $|s \mathscr{A} - s\mathscr{A}|$ using higher energy methods would provide stronger estimates than using iterative results like Lemma $\ref{gensum}$ along with a base case such as Theorem $\ref{highsum}$. 
\par

We have previously mentioned that an application of Cauchy-Schwarz inequality with Theorem $\ref{main}$ would give a bound of the shape
\[ |s \mathscr{A} - s\mathscr{A}| \gg_{\mathcal{C}} |A|^{3 - 4^{-s + 1}}. \]
Theorem $\ref{gensum}$ beats the above estimate by a factor of $|A|^{\delta' \cdot 4^{-s+1}}(\log|A|)^{ -C \cdot 4^{-s+3}}$ where $\delta' = 7/23$ and $C = 36/23$.
\par

We now focus our attention on applications of Theorem $\ref{genen}$. We note that in \S5, we proved Theorem $\ref{newthird}$, which stated that for $s \geq 3$, we have
\[ E_{s,3}(A) \ll_{\mathcal{C}} |A|^{s-1} E_{s-1,2}(A) \log |A| + |A|^{3s-6}. \]
Combining this with Theorem $\ref{genen}$, we get that 
\[ E_{s,3}(A) \ll_{\mathcal{C}} |A|^{s-1} |A|^{2s - 5 + (1/4 - c) \cdot 2^{-s + 5}}\log |A| = |A|^{3s-6 + (1/4 - c) \cdot 2^{-s + 5}} \log |A| , \]
where $c= 1/7246$ and $s \geq 5$. We record this as the following theorem. 

\begin{theorem} \label{newthirdnew}
Let $s \geq 5$ and let $c = 1/7246$. Then we have
\[E_{s,3}(A) \ll_{\mathcal{C}}  |A|^{3s-6 + (1/4 - c) \cdot 2^{-s + 5}} \log  |A| . \]
\end{theorem}

Moreover, we can combine Theorem $\ref{genen}$ along with Lemma $\ref{thistoo}$ to get
\[ E'_{2s,3}(A) \ll_{\mathcal{C}} |A|^{2s-1} E_{2s-1,2}(A) \log |A| + |A|^{6s-6} \ll_{\mathcal{C}} |A|^{6s-6 + (1/4 - c) \cdot 2^{-2s + 5}} \log |A|, \]
whenever $s \geq 3$. This, in turn, implies that
\[ T_{2s}(A) \ll_{\mathcal{C}} |A|^{6s-6 + (1/4 - c) \cdot 2^{-2s + 5}} \log |A|, \]
whenever $s \geq 3$. Recalling our setup from \S7, we see that if we employ the above estimate in place of $\eqref{thistoo2}$ in the proof of Theorem $\ref{highsum}$ to obtain
\[ |A|^{15/8} |3\mathscr{A} - 3 \mathscr{A}|^{11/8} |3 \mathscr{A} - 3 \mathscr{A}|^{3/2} (\log|A|)^{1/2} \gg_{\mathcal{C}} |A|^{21/2 - 1/8 + c/2} (\log|A|)^{-4}. \] 
Upon simplifying exponents, we get
\[ |3\mathscr{A} - 3 \mathscr{A}|^{23/8} \gg_{\mathcal{C}} |A|^{17/2 + c/2} (\log|A|)^{-9/2}, \]
which delivers the bound
\begin{equation} \label{itersum}
 |3\mathscr{A} - 3 \mathscr{A}| \gg_{\mathcal{C}}  |A|^{3 - 1/23+ 4c/23} (\log|A|)^{-36/23}, 
 \end{equation}
where $c = 1/7246$. This improves upon the bound provided by Theorem $\ref{highsum}$ by a factor of $|A|^{4c/23}$. We can substitute $\eqref{itersum}$ in place of Theorem $\ref{highsum}$ in the proof of Theorem $\ref{gensum}$ to get the following result. 

\begin{theorem} \label{noseat}
Let $s \geq 3$ and let $c = 1/7246$. Then we have
\[ |s \mathscr{A} - s\mathscr{A}| \gg_{\mathcal{C}} |A|^{ 3 - \delta \cdot 4^{-s+3} } (\log|A|)^{ -C \cdot 4^{-s+3}},\]
where $\delta = 1/23 - 4c/23$, and $C = 36/23$. 
\end{theorem}

In comparison to Theorem $\ref{gensum}$, the lower bound provided by Theorem $\ref{noseat}$ is stronger by a factor of $|A|^{c \cdot 4^{-s+4}23^{-1}}$, when $s \geq 3$. Furthermore, since Theorems $\ref{gensum}$ and $\ref{noseat}$ only treat the cases when $s \geq 3$, we see that the best known lower bound in the $s=2$ case still follows from Theorem $\ref{highsum}$.


\section{Proof of Theorem $\ref{kglw}$}

We will dedicate this section to finding upper bounds for $r_{s}(A)$ in terms of $|A|$, $\mathcal{C}$ and $s$. In particular, we will bound $r_{s}(A)$ in terms of $E_{s'}(A)$ for some $s' < s$. Thus, we prove the following result. 

\begin{lemma} \label{bantz}
Let $A$ be a finite subset of $I$, and let $p \geq 2$ be a natural number. Then we have
\begin{equation}  \label{bantz1}
 r_{2p+1}(\vec{n})  \ll_{\mathcal{C}}   |A|^{2p/3} E_{p,2}(A)^{2/3} + |A|^{2p-2},
 \end{equation}
and
\begin{equation}  \label{bantz2}
 r_{2p}(\vec{n})   \ll_{\mathcal{C}} E_{p,2}(A)^{1/3} E_{p-1,2}(A)^{1/3} |A|^{(2p-1)/3} + |A|^{2p-3}. 
 \end{equation}
\end{lemma}

\begin{proof}
We first prove $\eqref{bantz1}$. This is equivalent to showing that for every $\vec{n} \in (2p+1) \mathscr{A}$, we have
\[ r_{2p+1}(\vec{n}) \ll_{\mathcal{C}}  |A|^{2p/3} E_{p,2}(A)^{2/3} + |A|^{2p-2},  \]
where the implicit constant does not depend on $\vec{n}$. In this endeavour, we define for each $\vec{u} \in p \mathscr{A}$, the point 
\begin{equation} \label{huh1}
 p_{\vec{u}} = \vec{n} - \vec{u}, 
 \end{equation}
and the curve
\begin{equation} \label{huh2}
 l_{\vec{u}} = \{(t,\psi(t)) \ | \ t \in \mathbb{R} \} + \vec{u} . 
 \end{equation}
Next, we define the set of points $P$ as
\[ P = \{ p_{\vec{u}} \ | \ \vec{u} \in p \mathscr{A} \}, \]
and the set of curves $L$ as
\[ L = \{ l_{\vec{u}} \ | \ \vec{u} \in p \mathscr{A} \}. \]
Using Proposition $\ref{aslo}$, we note that $L$ is a set of $\mathcal{C}$-valid curves.
\par

Next, we observe that
\[ r_{2p+1}(\vec{n}) \leq \sum_{\vec{u}, \vec{v} \in p\mathscr{A}} \mathds{1}_{p_{\vec{u}} \in l_{\vec{v}}} r_{p} (\vec{u}) r_{p} (\vec{v}). \]
The right hand side above counts weighted incidences between $P$ and $L$, with both the sets having the same weight function $r_{p}$. Thus, we can use Lemma $\ref{wtst1}$ to obtain 
\begin{align*} \sum_{\vec{u}, \vec{v} \in p\mathscr{A}} \mathds{1}_{p_{\vec{u}} \in l_{\vec{v}}} r_{p} (\vec{u}) r_{p} (\vec{v}) \ll_{\mathcal{C}} \
 & (\sum_{\vec{u} \in p\mathscr{A}} r_{p}(\vec{u}))^{2/3}(\sum_{\vec{u} \in p\mathscr{A}} r_{p}(\vec{u})^2)^{2/3}\\ 
 & + r_{p}(A) \sum_{\vec{u} \in p\mathscr{A}} r_{p}(\vec{u}). \end{align*}
Combining the two preceding inequalities, we get
\[ r_{2p+1}(\vec{n})  \ll_{\mathcal{C}}   |A|^{2p/3} E_{p,2}(A)^{2/3} + |A|^{p-2}|A|^{p}, \]
which is the desired upper bound.
\par

We use a similar idea to prove $\eqref{bantz2}$. Given $\vec{n} \in 2p \mathscr{A}$, we use $P$ to denote the set of points
\[ P = \{ p_{\vec{u}} \ | \ \vec{u} \in p\mathscr{A} \}, \]
and $L$ to denote the set of $\mathcal{C}$-valid curves 
\[ L = \{ l_{\vec{v}} \ | \ \vec{v} \in (p-1) \mathscr{A} \}, \]
where $p_{\vec{u}}$ and $l_{\vec{v}}$ are defined in $\eqref{huh1}$ and $\eqref{huh2}$ respectively. As before, we observe that
\[ r_{2p}(\vec{n}) \leq \sum_{\substack{ \vec{u} \in p \mathscr{A}, \\ \vec{v} \in (p-1) \mathscr{A}}} \mathds{1}_{p_{\vec{u} \in l_{\vec{v}}}} r_{p}(\vec{u}) r_{p-1}(\vec{v}), \]
and thus we can use Lemma $\ref{wtst1}$ to deduce that
\[ r_{2p}(\vec{n})   \ll_{\mathcal{C}} E_{p,2}(A)^{1/3} E_{p-1,2}(A)^{1/3} |A|^{(2p-1)/3} + |A|^{2p-3}. \]
Hence, we conclude the proof of Lemma $\ref{bantz}$.  
\end{proof}

We remark that when $s \geq 6$, we can combine Lemma $\ref{bantz}$ with $\eqref{lopa}$ to bypass Lemma $\ref{slim}$, and use the subsequent estimates along with Theorem $\ref{main}$ to establish an alternative proof of Theorem $\ref{kigi}$. In particular, upon combining $\eqref{weonly}$ along with the incidence geometric inequality $\eqref{lopa}$, we can show that
\[ E_{s,3}(A) \ll_{\mathcal{C}} (E_{s-1,2}(A) + r_{s}(A)^2) |A|^{s-1} \log |A| + |A|^{3s-6}. \]
In the case when $s \geq 6$, we see that the above inequality amalgamates with Lemma $\ref{bantz}$ and Theorem $\ref{main}$ to deliver the bound
\[ E_{s,3}(A) \ll_{\mathcal{C}} |A|^{3s - 6 + 2^{-s+3}} \log |A|. \]
\par

We now incorporate estimates from Theorems $\ref{main}$ and $\ref{genen}$ along with Lemma $\ref{bantz}$ to prove Theorem $\ref{kglw}$. As before, let $c = 1/7246$. When $p = 2$, we have
\[ r_{5}(A) \ll_{\mathcal{C}} |A|^{4/3} E_{2,2}(A)^{2/3} + |A|^{2} \ll_{\mathcal{C}} |A|^{2 + 2/3}. \]
When $p = 3$, we use Theorem $\ref{main}$ to obtain
\[ r_{6}(A) \ll_{\mathcal{C}} |A|^{1 + 1/6} |A|^{2/3} |A|^{5/3} + |A|^{3} \ll_{\mathcal{C}} |A|^{ 3 + 1/2}, \]
and 
\[ r_{7}(A) \ll_{\mathcal{C}} |A|^2 |A|^{2 + 1/3} + |A|^{4} \ll_{\mathcal{C}} |A|^{4 + 1/3}. \]
We can utilise Theorem $\ref{genen}$ to estimate $r_{8}(A)$. In particular, we see that
\[ r_{8}(A) \ll_{\mathcal{C}} E_{4,2}(A)^{1/3} E_{3,2}(A)^{1/3} |A|^{7/3} + |A|^{5} \ll_{\mathcal{C}} |A|^{  5 + 1/4 - c/3  }. \]
For $p \geq 4$, upon applying Theorem $\ref{genen}$ along with Lemma $\ref{bantz}$, we get
\[ r_{2p+1}(A) \ll_{\mathcal{C}}  |A|^{2p/3} E_{p,2}(A)^{2/3} + |A|^{2p-2} \ll_{\mathcal{C}} |A|^{2p - 2 + (1/4 - c) \cdot 2^{-p + 5} \cdot 3^{-1} }. \]
Similarly when $p \geq 5$, we find that
\[ r_{2p}(A) \ll_{\mathcal{C}} E_{p,2}(A)^{1/3} E_{p-1,2}(A)^{1/3} |A|^{(2p-1)/3} + |A|^{2p-3} \ll_{\mathcal{C}} |A|^{2p - 3 + (1/4 - c) \cdot 2^{- p + 4} }. \]
This concludes the proof of Theorem $\ref{kglw}$.


\section{Proofs of Theorems $\ref{mainth1}$ and $\ref{mainth2}$}

We utilise this section to prove Theorems $\ref{mainth1}$ and $\ref{mainth2}$. We note that all the estimates that we have proven till now have been for intervals $I$, continuous functions $\psi : I \to \mathbb{R}$ and parameters $\mathcal{C} > 0$ that satisfy $\eqref{fis1}$ and $\eqref{fis2}$. We will now use these estimates to deduce our main results. 
\par

We being by proving the following proposition which gives us a class of functions and intervals that satisfy $\eqref{fis1}$ and $\eqref{fis2}$. 

\begin{Proposition} \label{convexity}
Let $I$ be an interval, and let $\psi: I \to \mathbb{R}$ be a function such that $\psi$ and $\psi'$ are continuous and differentiable on $I$, and $\psi''(x) \neq 0$ for any $x \in I$. Then $I$ and $\psi$ satisfy $\eqref{fis1}$ and $\eqref{fis2}$ with some absolute constant $\mathcal{C} \ll 1$.
\end{Proposition}

\begin{proof}
We note that in order to verify $\eqref{fis1}$, it suffices to show that for every pair $\delta_1, \delta_2$ of real numbers such that $\delta_1 \neq 0$, the equation 
\begin{equation} \label{conv1}
 \psi(x) - \psi(x- \delta_1) - \delta_2 = 0 
 \end{equation}
has $O(1)$ number of solutions $x \in I' = I \cap (I + \{ \delta_1\})$. Let $a_1, a_2 \in I'$ be distinct elements such that $a_i$ satisfies $\eqref{conv1}$ for $1 \leq i \leq 2$. Applying Rolle's theorem on the function $\psi(x) - \psi(x - \delta_1)$, we deduce that there exists $b \in I'$ such that
\[  \psi'(b) = \psi'(b - \delta_1) . \]
Since $\delta_1 \neq 0$, a second application of Rolle's theorem for the function $\psi'$ implies that there exists $c \in I'$ such that
\[ \psi''(c) = 0, \]
which contradicts our hypothesis that $\psi''(x) \neq 0$ for all $x \in I$. Thus, there exists at most one solution to the system of equations
\[ x - y - \delta_1 = \psi(x) - \psi(y) - \delta_2 = 0. \] 
\par

Similarly, we see that verifying inequality $\eqref{fis2}$ reduces to showing that for every pair $n_1,n_2$ of real numbers, the equation 
\begin{equation} \label{conv2}
 \psi(x) + \psi(n_1 - x) - n_2 = 0 
 \end{equation}
has $O(1)$ number of solutions $x \in I' = I \cap (-I + \{ n_1\})$. Note that if $x$ satisfies $\eqref{conv2}$, then so does $n_1 - x$, whence, we can assume that $x \leq n_1/2$. As before, let $a_1, a_2 \in I'$ be distinct elements satisfying the above equation and the condition that $a_1 < a_2 \leq n_1/2$. Using Rolle's theorem, we see that there exists $b \in I'$ such that $b \in (a_1, a_2)$ and
\[ \psi'(b) = \psi'(n_1 - b) . \]
Since $b < a_2 \leq n_1/2$, we see that $b \neq n_1 - b$. Thus, applying Rolle's theorem once more, we deduce the existence of some $c \in I'$ such that
\[ \psi''(c) = 0, \]
which contradicts our hypothesis that $\psi''(x) \neq 0$ for all $x \in I$. Hence, there are at most $O(1)$ solutions to the system of equations
\[ x + y - n_1 = \psi(x) + \psi(y) - n_2 = 0, \]
whereupon, we conclude the proof of Proposition $\ref{convexity}$.  
\end{proof}  

Thus, we let $\psi$ be a polynomial with real coefficients and degree $d \geq 2$, and we let $A$ be some large, finite, non-empty subset of real numbers. Our strategy is to write $A$ as a union of a small number of sets $A_i$ such that for each $i$, either the set $A_i$ is contained in a suitable interval $I_i$ where we can apply the estimates that we have proven in \S2-10, or the set $A_i$ is small. This is recorded in the following lemma.

\begin{lemma} \label{prelim2}
Let $\psi$ be a polynomial with real coefficients and degree $d \geq 2$, and let $A$ be a finite, non-empty subset of real numbers. Then, we can write 
\[ A = A_0 \cup A_1 \dots A_{r} \cup A_{r+1}, \]
where $r = O_{d}(1)$ and $|A_{r+1}| = O_{d}(1)$, and for each $0 \leq i < j \leq r+1$, the sets $A_i$ and $A_j$ are disjoint. Moreover, for each $0 \leq i \leq r$, the set $A_i$ is contained in an interval $I_i$ such that $\psi$ and $I_i$ satisfy $\eqref{fis1}$ and $\eqref{fis2}$ with some absolute constant $\mathcal{C} \ll 1$.
\end{lemma}

\begin{proof}
We define $E$ to be the set of real zeroes of the function $\psi''$. Since $\psi$ is a polynomial of degree $d \geq 2$, the functions $\psi'$ and $\psi''$ are polynomials of degree $d-1$ and $d-2$ respectively. In particular, this implies that $|E| = O_{d}(1)$. First, we assume that $E$ is non-empty, in which case, we write 
\[ E = \{ e_1, e_2, \dots, e_{r} \} \ \text{such that} \ e_1 \leq e_2 \leq \dots \leq e_{r},  \] 
where $r = |E|$. We let
\[ I_{i} = (e_i, e_{i+1}) \ \text{for each} \ 1 \leq i \leq r-1 , \]
while we denote $I_0 =  (-\infty, e_1)$, and $I_{r} =  (e_{r}, \infty)$. We finish our set-up by letting $A_i = A \cap I_i$ for each $0 \leq i \leq r$, and we set $A_{r+1} = A \cap E$. If, on the other hand, $E$ is empty, we simply set $r=0$ and $A_0 = A$. In either case, we note that for each $0 \leq i \leq r$, Proposition $\ref{convexity}$ implies that the interval $I_{i}$ and the function $\psi$ satisfy $\eqref{fis1}$ and $\eqref{fis2}$ with some absolute constant $\mathcal{C} \ll 1$. Consequently, we conclude the proof of Lemma $\ref{prelim2}$.
\end{proof}

We can already prove Theorem $\ref{mainth2}$ by utilising Lemma $\ref{prelim2}$ along with Theorems $\ref{highsum}$ and $\ref{noseat}$. In particular, let $\psi$ be a polynomial with real coefficients and degree $d \geq 2$, and let $A$ be some finite, non-empty set of real numbers. Using Lemma $\ref{prelim2}$ and the pigeonhole principle, we see that there exists some natural number $r = O_{d}(1)$, and some $0 \leq i \leq r$ such that
\[ |A_i| \geq (|A| - |A_{r+1}|) (r+1)^{-1} \gg_{d} |A|. \]
Moreover, $A_i$ lies in an interval $I$ such that $\psi$ and $I_i$ satisfy $\eqref{fis1}$ and $\eqref{fis2}$ with some absolute constant $\mathcal{C} \ll 1$. Thus, setting $\mathscr{A}_i = \{ (a, \psi(a)) \ | \ a \in A_i \}$, we use Theorem $\ref{highsum}$ to deduce that
\begin{align*}
 |2 \mathscr{A} - 2 \mathscr{A}|  \geq |2 \mathscr{A}_i - 2 \mathscr{A}_i| & \gg |A_i|^{3 - 2/11}(\log |A_i|)^{-18/11}  \\
&\gg_{d}  |A|^{3 - 2/11}(\log |A|)^{-18/11} . 
\end{align*}
Similarly, Theorem $\ref{noseat}$ implies that for every natural number $s \geq 3$, we have
\begin{align*}
 |s \mathscr{A} - s \mathscr{A}| \geq |s \mathscr{A}_i - s \mathscr{A}_i| & \gg |A_i|^{ 3 - \delta \cdot 4^{-s+3} } (\log|A_i|)^{ -C \cdot 4^{-s+3}} \\
 & \gg_{d,s}  |A|^{ 3 - \delta \cdot 4^{-s+3} } (\log|A|)^{ -C \cdot 4^{-s+3}}   ,
 \end{align*}
where $\delta = 1/23 - 4c/23$, and $C = 36/23$. Thus, we finish the proof of Theorem $\ref{mainth2}$. 
\par

In order to prove Theorem $\ref{mainth1}$, we need to record one more preparatory result. Moreover, we prove this result in a slightly more general context since we intend to utilise it in the next section for more general systems of equations. Thus, we let $f$ and $g$ be polynomials of degree $d_1$ and $d_2$ respectively such that $d_2 > d_1 \geq 1$. Let $A$ be a finite, non-empty subset of real numbers. For each natural number $s$, we define $E_{f,g,s,2}(A)$ to be the number of solutions to the system of equations
\[ \sum_{i=1}^{s} (f(x_i) - f(x_{i+s})) = \sum_{i=1}^{s}( g(x_i) - g(x_{i+s})) = 0, \]
with $x_i \in A$ for each $1 \leq i \leq 2s$. When $A$ is the empty set, we define $E_{f,g,s,2}(A) = 0$.

\begin{lemma} \label{prelim3}
Let $s$ be a natural number, and let $f$ and $g$ be polynomials of degree $d_1$ and $d_2$ respectively such that $d_2 > d_1 \geq 1$. Moreover, let $A$ be a finite, non-empty set of real numbers such that $A = \cup_{i=0}^{r+1} A_i$, where $A_i$ and $A_j$ are pairwise disjoint whenever $0 \leq i < j \leq r+1$. Then we have
\[ E_{f,g,s,2}(A)   \ll (r+2)^{2s-1} \sum_{i=0}^{r+1} (E_{f,g,s,2}(A_i) +  |A|^{-2s}) + |A|^{-2s}.   \] 
\end{lemma}

\begin{proof}
For every $\vec{\alpha} = (\alpha_1, \alpha_2) \in \mathbb{R}^2$, we define 
\[ \mathfrak{f}(\vec{\alpha}) = \sum_{a \in A} e(f(a) \alpha_1 + g(a) \alpha_2), \]
and 
\[ \mathfrak{f}_{i}(\vec{\alpha}) = \sum_{a \in A_i} e( f(a) \alpha_1 + g(a) \alpha_2) \ \text{for each} \ 0 \leq i \leq r+1, \]
where for each $\theta \in \mathbb{R}$, we use the notation $e(\theta) = e^{2 \pi i \theta}$. If for some $0 \leq i \leq r+1$, the set $A_i$ is empty, we set $\mathfrak{f}_{i}(\vec{\alpha}) = 0$. For ease of notation, given $\vec{a} = (a_1, \dots, a_{2s}) \in A^{2s}$, we write
\[ \vec{\xi}_{\vec{a}} = (\xi_{1,\vec{a}}, \xi_{2, \vec{a}}) = \Big( \sum_{i=1}^{s}( f(a_{i}) - f(a_{i+s})), \sum_{i=1}^{s}( g(a_{i}) - g(a_{i+s})) \Big). \] 
We further define
\[ M_1 = \min_{\substack{ \vec{a} \in A^{2s} \\ \xi_{1,\vec{a}} \neq 0}} |\xi_{1,\vec{a}}|, \ \text{and} \ M_2 =  \min_{\substack{ \vec{a} \in A^{2s} \\ \xi_{2,\vec{a}} \neq 0}} |\xi_{2,\vec{a}}| . \]
We set
\[ X =  |A|^{4s} \ceil{ ( 1 + M_1^{-2} + M_2^{-2} )}. \]
\par

We first note that for any $\xi \in \mathbb{R}$ such that $\xi \neq 0$, we have
\[  \Big| \int_{[0,X]} e( \alpha \xi) \ d{\alpha} \ \Big| = \Big| (2\pi i \xi)^{-1} (e(X \xi) - 1) \Big| \ll |\xi|^{-1}. \]
Similarly, for any $\vec{\xi} = (\xi_1, \xi_2) \in \mathbb{R}^2$ such that $\vec{\xi} \neq (0,0)$, we have
\begin{equation} \label{gch}
 \Big| \int_{[0,X]^{2}} e( \alpha_1 \xi_1 + \alpha_2 \xi_2) \ d \vec{\alpha} \ \Big| \ll X \sup_{\substack{1 \leq i \leq 2 \\ \xi_i \neq 0}} |\xi_i|^{-1}. 
 \end{equation}
Moreover, when $\vec{\xi} = (0,0)$, we find that
\[ \Big| \int_{[0,X]^{2}} e( \alpha_1 \xi_1 + \alpha_2 \xi_2) \ d \vec{\alpha} \ \Big|  = X^2 . \]
Thus, we see that
\begin{align*}
X^{-2} \int_{[0,X]^{2}} |\mathfrak{f}(\vec{\alpha})|^{2s} \ d\vec{\alpha} 
& = X^{-2} \sum_{\vec{a} \in A^{2s}}  \int_{[0,X]^{2}} e( \alpha_1 \xi_{1,\vec{a}} + \alpha_2 \xi_{2, \vec{a}}) \ d \vec{\alpha} \ \\
& = \sum_{\substack{   \vec{a} \in A^{2s} \\   \vec{\xi}_{\vec{a}}   = (0,0) }} 1  +  X^{-2} \sum_{\substack{   \vec{a} \in A^{2s} \\   \vec{\xi}_{\vec{a}}   \neq (0,0) }}  \int_{[0,X]^{2}} e( \alpha_1 \xi_{1,\vec{a}} + \alpha_2 \xi_{2, \vec{a}}) \ d \vec{\alpha} \ .
 \end{align*}
 \par
 
Noting $\eqref{gch}$ and the fact that
 \[ \sum_{\substack{   \vec{a} \in A^{2s} \\   \vec{\xi}_{\vec{a}}   = (0,0) }} 1 = E_{f,g,s,2}(A), \]
the preceding expression gives us
 \begin{align*}
 E_{f,g,s,2}(A)  -  X^{-2} \int_{[0,X]^{2}} |\mathfrak{f}(\vec{\alpha})|^{2s} \ d\vec{\alpha}  
& \ll  X^{-2}  X   \sum_{\substack{   \vec{a} \in A^{2s} \\   \vec{\xi}_{\vec{a}}   \neq (0,0) }}  \sup_{\substack{1 \leq i \leq 2 \\ \xi_{i, \vec{a}} \neq 0}} |\xi_{i,\vec{a}}|^{-1} \\
 & \ll X^{-1} |A|^{2s} (M_1^{-1} + M_2^{-1})  \ll |A|^{-2s}. 
\end{align*}
Similarly, we have
\[ X^{-2} \int_{[0,X]^{2}} |\mathfrak{f}_i(\vec{\alpha})|^{2s} \ d\vec{\alpha}  - E_{f,g,s,2}(A_i) \ll |A|^{-2s}. \]
Finally, applying H\"older's inequality, we see that
\[  |\mathfrak{f}(\vec{\alpha})|^{2s} = |\sum_{i=0}^{r+1} \mathfrak{f}_{i}(\vec{\alpha})  |^{2s}  \leq (r+2)^{2s-1} \sum_{i=0}^{r+1}   | \mathfrak{f}_{i}(\vec{\alpha})  |^{2s}, \]
from which we deduce that
\[  X^{-2} \int_{[0,X]^{2}} |\mathfrak{f}(\vec{\alpha})|^{2s} d \vec{\alpha} \leq  (r+2)^{2s-1} \sum_{i=0}^{r+1}  X^{-2} \int_{[0,X]^{2}}  | \mathfrak{f}_{i}(\vec{\alpha})  |^{2s} d\vec{\alpha} . \]
Combining the preceding inequalities, we find that
\[ E_{f,g,s,2}(A)   \ll (r+2)^{2s-1} \sum_{i=0}^{r+1} (E_{f,g,s,2}(A_i) +  |A|^{-2s}) + |A|^{-2s},  \] 
which is the desired conclusion.
\end{proof}

It is worth noting that in the conclusion of Lemma $\ref{prelim3}$, the $O(|A|^{-2s})$ factors can be removed as well as the implicit multiplicative factor in the Vinogradov notation can be made explicit by letting $X \to \infty$ in the above proof, but for our purposes, the conclusion of Lemma $\ref{prelim3}$, as recorded, is sufficient. 
\par

We now proceed with the proof of Theorem $\ref{mainth1}$. Thus, let $\psi$ be a polynomial with real coefficients and degree $d \geq 2$, and let $A$ be a finite, non-empty set of real numbers. We use Lemma $\ref{prelim2}$ to write 
\[ A = A_0 \cup A_1 \dots A_{r} \cup A_{r+1}. \]
Since for each $0 \leq i \leq r$, the set $A_i$ is contained in an interval $I_i$ such that $\psi$ and $I_i$ satisfy $\eqref{fis1}$ and $\eqref{fis2}$ with some absolute constant $\mathcal{C} \ll 1$, we can use Theorem $\ref{main}$ to deduce that
\[ E_{3,2}(A_i) \ll |A_i|^{3 + 1/2} \leq |A|^{3+1/2}. \]
Moreover, we have
\[ E_{3,2}(A_{r+1}) \leq |A_{r+1}|^6 = O_{d}(1). \]
We combine the preceding inequalities with Lemma $\ref{prelim3}$ to obtain
\begin{align*}  E_{3,2}(A)  & \ll (r+2)^{5} \sum_{i=0}^{r+1} (E_{3,2}(A_i) +  |A|^{-6}) + |A|^{-6}\\
& \ll_{d} (r+2)^{6} |A|^{3+1/2} + (r+4)^{6}|A|^{-6} + 1 \ll_{d} |A|^{3 + 1/2}. 
\end{align*}
\par

Similarly, when $s \geq 4$, we use Theorem $\ref{genen}$ to see that
\[ E_{s,2}(A_i) \ll |A|^{2s - 3 + (1/4 - c) \cdot 2^{-s+4}},\]
whenever $0 \leq i \leq r$. Furthermore, we use the trivial bound
\[ E_{s,2}(A_{r+1}) \leq |A_{r+1}|^{2s} = O_{s,d}(1). \]
Incorporating these estimates with Lemma $\ref{prelim3}$, we find that
\begin{align*}  E_{s,2}(A)  & \ll (r+2)^{2s-1} \sum_{i=0}^{r+1} (E_{s,2}(A_i) +  |A|^{-2s}) + |A|^{-2s}\\
& \ll_{d,s} (r+2)^{2s} |A|^{2s - 3 + (1/4 - c) \cdot 2^{-s+4}} + (r+4)^{2s}|A|^{-2s} + 1 \\
& \ll_{d,s} |A|^{2s - 3 + (1/4 - c) \cdot 2^{-s+4}}. 
\end{align*}
Thus, we finish the proof of Theorem $\ref{mainth1}$.


\section{Generalisations}

In this section, we devote our attention to systems of the form 
\begin{equation} \label{frk}
 \sum_{i=1}^{s} (f(x_i) - f(x_{i+s})) = \sum_{i=1}^{s}( g(x_i) - g(x_{i+s})) = 0, 
 \end{equation}
where functions $f$ and $g$ are linearly independent. As before, given a finite, non-empty subset $A$ of $\mathbb{R}$, we denote $E_{f,g,s,2}(A)$ to be the number of solutions to $\eqref{frk}$ such that $x_i \in A$ for each $1 \leq i \leq 2s$. In this section, we prove estimates akin to Theorem $\ref{mainth1}$ for $E_{f,g,s,2}(A)$ when $f$ and $g$ are polynomials of differing degrees. This will be recorded as Theorem $\ref{genz}$. Moreover, we briefly comment on other systems of equations that are amenable to our methods. 
\par

We begin with some preliminary manoeuvres. We note that 
\[ \sum_{i=1}^{s} (f(x_i) - f(x_{i+s})) =  \sum_{i=1}^{s} ((f(x_i) - c)- (f(x_{i+s})-c)) , \]
for each $c \in \mathbb{R}$. Thus, if $f$ and $g$ are polynomials and $A$ is some finite, non-empty subset of $\mathbb{R}$, and we want to estimate $E_{f,g,s,2}(A)$, we can begin by assuming that $f(0) = g(0) = 0$. Moreover, we note that if $\eqref{frk}$ holds for some choice of $x_1, \dots, x_{2s}$, then we also have
\[  \sum_{i=1}^{s} (h_1(x_i) - h_1(x_{i+s})) = \sum_{i=1}^{s}( h_2(x_i) - h_2(x_{i+s})) = 0, \]
where
\[ h_1(x) = \alpha_1 f(x) + \alpha_2 g(x) \ \text{and} \ h_2(x) = \beta_1 f(x) + \beta_2 g(x) \]
for any $\alpha_1, \alpha_2, \beta_1, \beta_2 \in \mathbb{R}$. Hence, if $f$ and $g$ are polynomials with real coefficients, we can assume that $f$ and $g$ have different degrees.
\par

An instance of this situation is the system of equations
\[ \sum_{i=1}^{s} (x_i^3 - x_{i+s}^3) = \sum_{i=1}^{s} (x_i^2 - x_{i+s}^2) = 0,  \]
with $x_i \in A$ for $1 \leq i \leq 2s$, where $A$ is some non-empty, finite subset of $(0, \infty)$. We note that this particular example is equivalent to counting the number of solutions to 
\[ \sum_{i=1}^{s} (x_i^{3/2} - x_{i+s}^{3/2}	) = \sum_{i=1}^{s} (x_i - x_{i+s}) = 0,  \]
with $x_i \in A' \subseteq (0, \infty)$ for $1 \leq i \leq 2s$, and $A' = \{ a^{2} \ | \ a \in A \}. $
We can bound the latter using Proposition $\ref{convexity}$ and Theorem $\ref{noseat}$.
\par

We now return to $\eqref{frk}$ where $f$ and $g$ are arbitrary polynomials with real coefficients and degrees $d_1$ and $d_2$ respectively, such that $d_1 \neq d_2$. Without loss of generality, we can assume that $d_2 > d_1\geq 2$, since the case $d_1 = 1$ reduces to Theorem $\ref{mainth1}$. We first consider the case when our set $A$ is contained in some interval $I$ satisfying 
\begin{equation} \label{postfor}
 g(x)\cdot g'(x) \cdot f(x) \cdot f'(x) \cdot (g''(x) f'(x) - f''(x) g'(x)) \neq 0 \ \text{for each} \ x \in I . 
\end{equation}
Since each of $f,f',g,g'$ is a polynomial of degree at least one, and the polynomial
\[g''(x) f'(x) - f''(x) g'(x) \] 
has $d_1d_2(d_2 - d_1)x^{d_1 + d_2 - 3}$ as the leading term, the expression on the left hand side of $\eqref{postfor}$ has finitely many real roots. This implies that there does exist at least one interval $I$ satisfying $\eqref{postfor}$. With these assumptions in hand, we can prove an analogue of Theorem $\ref{main}$ for the system $\eqref{frk}$. 

\begin{theorem} \label{propn}
Let $s \geq 3$ be a natural number, and let $f$ and $g$ be polynomials of degrees $d_1$ and $d_2$ respectively, where $d_2 > d_1 \geq 2$. Moreover, let $I$ be an interval satisfying $\eqref{postfor}$ and let $A$ be some finite, non-empty subset of $I$. Then we have
\[ E_{f,g,s,2}(A) \ll |A|^{2s-3 + \eta_{s}}, \]
where $\eta_{3} = 1/2$, and $\eta_{s} = (1/4 - c) \cdot 2^{-s+4}$ whenever $s \geq 4$, where $c = 1/7246$.
\end{theorem}

\begin{proof}
Since $f$ is monotone in the interval $I$, we see that $E_{f,g,s,2}(A)$ is equivalent to counting the number of solutions $E_{h,s,2}(A)$ to the system of equations
\[  \sum_{i=1}^{s} (u_i - u_{i+s})  = \sum_{i=1}^{s}( h(u_i) - h(u_{i+s})) = 0, \]
such that $u_i \in A' \subseteq I'$ for each $1 \leq i \leq 2s$, where $h(u) = g(f^{-1}(u))$ for each $u \in I'$, and 
\[ A' = \{ f(a) \ | \ a \in A \} \ \text{and} \ I' = \{ f(x) \ | \ x \in I \}. \]
Moreover, we have $|A'| = |A|$. 
\par

As $f$ is a polynomial with a non-zero derivative in $I$, we can deduce that $f^{-1}$ is a continuous and differentiable function on $I'$ satisfying
\[ (f^{-1})'(u) = 1/f'(f^{-1}(u)) \ \text{for each} \ u \in I'. \]
Furthermore, this implies that $h$ is continuous and differentiable on $I'$ satisfying
\[ h'(u) = g'(f^{-1}(u))/f'(f^{-1}(u)) \ \text{for each} \ u \in I' . \]
Differentiating once again, we find that
\[ h''(u) = \frac{1}{(f'(f^{-1}(u)))^3} (g''(f^{-1}(u)) f'(f^{-1}(u)) - f''( f^{-1}(u)) g'(f^{-1}(u))), \]
for each $u \in I'$. Noting $\eqref{postfor}$, we deduce that $h'(u) h''(u) \neq 0$ for each $u \in I'$. Moreover, we see that $h'$ is continuous and differentiable on $I'$. We now combine Proposition $\ref{convexity}$ along with Theorem $\ref{main}$ when $s=3$, and Theorem $\ref{genen}$ when $s \geq 4$, to obtain the desired bound for $E_{h,s,2}(A)$. This, in turn, provides the required estimate for $E_{f,g,s,2}(A)$, and so, we conclude the proof of Theorem $\ref{propn}$. 
\end{proof}

We finish this section by providing estimates for $E_{f,g,s,2}(A)$ when $A$ is any finite, non-empty subset of $\mathbb{R}$, instead of restricting $A$ to be a subset of some interval $I$ with suitable properties. 

\begin{theorem} \label{genz}
Let $A$ be some finite, non-empty subset of $\mathbb{R}$, let $s \geq 3$ be some natural number, and let $f$ and $g$ be polynomials of degrees $d_1$ and $d_2$ respectively, where $d_2 > d_1 \geq 1$. Then we have
\[ E_{f,g,s,2}(A) \ll_{d_2,d_1,s} |A|^{2s - 3 + \eta_{s}}, \]
where $\eta_{3} = 1/2$, and $\eta_{s} = (1/4 - c) \cdot 2^{-s+4}$ whenever $s \geq 4$, where $c = 1/7246$.
\end{theorem}

\begin{proof}

As we remarked previously, it suffices to consider the case when $d_1 \geq 2$, since the case when $d_1 = 1$ is handled by Theorem $\ref{mainth1}$. 
\par

We begin by defining the set
\[ E = \{x \in \mathbb{R} \ | \ g(x)\cdot g'(x) \cdot f(x) \cdot f'(x) \cdot (g''(x) f'(x) - f''(x) g'(x)) = 0 \}. \]
As before, we note that $|E|= O_{d_2}(1)$. If $E$ is empty, we can use Theorem $\ref{propn}$ to get the desired estimates for $E_{f,g,s,2}(A)$. Hence, we can assume that $E$ is non-empty, in which case, we write
\[ E = \{ e_1, e_2, \dots, e_{r} \} \ \text{such that} \ e_1 \leq e_2 \leq \dots \leq e_{r},  \] 
where $r = |E|$. For each $1 \leq i \leq r-1$, we define the interval $I_{i} = (e_{i}, e_{i+1})$. Furthermore, we set $I_{0} = (-\infty, e_1)$ and $I_{r} = (e_{r}, \infty)$. For each $0 \leq i \leq r$, we write $A_i = A \cap I_{i}$, and we denote $A_{r+1} = A \cap E$. 
\par

Noting Lemma $\ref{prelim3}$, we find that
\[ E_{f,g,s,2}(A)   \ll (r+2)^{2s-1} \sum_{i=0}^{r+1} (E_{f,g,s,2}(A_i) +  |A|^{-2s}) + |A|^{-2s}.   \] 
For each $0 \leq i \leq r$, we can use Theorem $\ref{propn}$ to deduce that 
\[ E_{f,g,s,2}(A_{i}) \ll |A|^{2s - 3 + \eta_{s}}. \]
Furthermore, we have
\[ E_{f,g,s,2}(A_{r+1}) \leq |A_{r+1}|^{2s} \leq r^{2s}.\]
Putting these estimates together, we see that
\[ E_{f,g,s,2}(A) \ll (r+2)^{2s}   |A|^{2s - 3 + \eta_{s}} + (r+2)^{2s} r^{2s} + (r+2)^{2s} |A|^{-2s}. \]
We recall that $r= O_{d_2,s}(1)$, whence
\[ E_{f,g,s,2}(A) \ll_{d_1,d_2,s} |A|^{2s - 3 + \eta_{s}}, \]
with which, we conclude the proof of Theorem $\ref{genz}$.  
\end{proof}


\appendix
\section{Proof of Lemma $\ref{wtst1}$}

In this section, we record the proof of Lemma $\ref{wtst1}$ as stated in \cite[Theorem 48]{Lu2017}. Let $\mathcal{D} > 0$ be some constant. We begin by stating Szemer\'{e}di-Trotter theorem for points and $\mathcal{D}$-valid curves as written in $\mathbb{R}^2$ \cite[Theorem 1.1]{PS1998}.

\begin{lemma} \label{ps1998}
Let $P$ be a set of points in $\mathbb{R}^2$, and let $L$ be a set of simple $\mathcal{D}$-valid curves all lying in $\mathbb{R}^2$. Then we have
\[ \sum_{p \in P} \sum_{l \in L} \mathds{1}_{p \in l} \ll_{\mathcal{D}} |P|^{2/3} |L|^{2/3} + |P| + |L|. \]
\end{lemma}

\begin{proof}[Proof of Lemma $\ref{wtst1}$]
Let $P$ be a set of points in $\mathbb{R}^2$, let $L$ be a set of $\mathcal{D}$-valid curves in $\mathbb{R}^2$ and let $w, w'$ be a weight functions on $L$ and $P$ respectively. Moreover, we have
\[ I_{w,w'}(P,L) = \sum_{p \in P} \sum_{l \in L} \mathds{1}_{p \in l} w'(p)w(l). \]
Let $J$ be the largest natural number such that $2^{J} \leq \n{P}_{\infty}$, and similarly, let $K$ be the largest natural number such that $2^K \leq  \n{L}_{\infty}$. Furthermore, for each $1 \leq j \leq J$, we write
\[ P_{j} = \{ p \in P \ | \ 2^{j} \leq w'(p) < 2^{j+1} \}, \]
and for each $1 \leq k \leq K$, we write
\[ L_{k} = \{ l \in L \ | \ 2^{k} \leq w(l) < 2^{k+1} \}. \]
\par

We see that
\begin{align*}
 \sum_{p \in P} \sum_{l \in L} \mathds{1}_{p \in l} w'(p)w(l) & = \sum_{j=0}^{J} \sum_{k=0}^{K} \sum_{p \in P_{j}}\sum_{l \in L_{k}}\mathds{1}_{p \in l} w'(p)w(l)  \ll \sum_{j=0}^{J} \sum_{k=0}^{K} 2^{j}2^{k} \sum_{p \in P_{j}}\sum_{l \in L_{k}}\mathds{1}_{p \in l}  \\
& \ll_{\mathcal{D}} \sum_{j=0}^{J} \sum_{k=0}^{K} 2^{j}2^{k} ( (|P_{j}||L_{k}|)^{2/3} + |P_{j}| + |L_{k}|),
 \end{align*}
 where the last inequality follows from Lemma $\ref{ps1998}$. We note that
 \[ \sum_{j=0}^{J} \sum_{k=0}^{K} 2^{j}2^{k} |P_{j}| = \bigg(\sum_{j=0}^{J} 2^{j}|P_{j}| \bigg) \bigg(\sum_{k=0}^{K} 2^{k}\bigg) \ll \n{P}_1 2^{K} \ll \n{P}_1 \n{L}_{\infty}, \]
and
 \[\sum_{j=0}^{J} \sum_{k=0}^{K} 2^{j}2^{k} |L_{k}| = \bigg(\sum_{j=0}^{J} 2^{j}\bigg) \bigg(\sum_{k=0}^{K} 2^{k} |L_{k}| \bigg)  \ll 2^{J} \n{L}_{1} \ll \n{P}_{\infty} \n{L}_{1}. \]
Moreover, we have
 \[  \sum_{j=0}^{J} \sum_{k=0}^{K} 2^{j}2^{k}  (|P_{j}||L_{k}|)^{2/3} = \bigg(\sum_{j=0}^{J} 2^{j} |P_{j}|^{2/3} \bigg) \bigg(\sum_{k=0}^{K} 2^{k} |L_{k}|^{2/3} \bigg). \]
 \par
 
We set $U = \{ 0 \leq j \leq J \ | \ 2^{j} \leq \n{P}_2^{2} \n{P}_1^{-1} \}$, and $V = \{0,1, \dots, J\} \setminus U$. We note that
\begin{align*}
 \sum_{j \in U} 2^{j} |P_{j}|^{2/3} & = \sum_{j \in U} 2^{j/3} ( 2^{j} |P_{j}|)^{2/3} \ll \n{P}_1^{2/3} \sum_{j \in U} 2^{j/3} \\
&    \ll \n{P}_1^{2/3} \n{P}_{2}^{2/3} \n{P}_1^{-1/3} = ( \n{P}_1 \n{P}_2^{2})^{1/3}.
 \end{align*}
 Furthermore, we have
\begin{align*}
\sum_{j \in V} 2^{j} |P_{j}|^{2/3} & = \sum_{j \in V} 2^{-j/3} (2^{2j} |P_{j}|)^{2/3}  \ll \n{P}_2^{4/3} \sum_{j \in V} 2^{-j/3} \\
& \ll \n{P}_2^{4/3} (\n{P}_2^{2} \n{P}_1^{-1})^{-1/3} \ll (\n{P}_2^{2} \n{P}_1)^{1/3}. 
\end{align*}
Thus, we deduce that
\[ \sum_{j=0}^{J} 2^{j} |P_{j}|^{2/3} = \sum_{j \in U} 2^{j} |P_{j}|^{2/3} + \sum_{j \in V} 2^{j} |P_{j}|^{2/3} \ll ( \n{P}_1 \n{P}_2^{2})^{1/3}. \]
Similarly, we find that
\[ \sum_{k=0}^{K} 2^{k} |L_{k}|^{2/3} \ll (\n{L}_1 \n{L}_2^2)^{1/3}. \]
\par

Thus, we combine the preceding inequalities to infer that
\[  \sum_{j=0}^{J} \sum_{k=0}^{K} 2^{j}2^{k}  (|P_{j}||L_{k}|)^{2/3} \ll (\n{P}_1 \n{P}_2^{2}\n{L}_1 \n{L}_2^2)^{1/3}, \]
which in turn implies that
\[ I_{w,w'}(P,L) \ll_{\mathcal{D}} (\n{P}_1 \n{P}_2^{2}\n{L}_1 \n{L}_2^2)^{1/3} + \n{P}_1 \n{L}_{\infty} + \n{P}_{\infty} \n{L}_{1}. \]
Hence, we conclude the proof of Lemma $\ref{wtst1}$. 
\end{proof}


\bibliographystyle{amsbracket}
\providecommand{\bysame}{\leavevmode\hbox to3em{\hrulefill}\thinspace}

\end{document}